\documentclass[11pt,english]{article}
\usepackage[letterpaper]{geometry}
\usepackage{amsmath,amssymb,amsthm}
\usepackage{graphicx}
\usepackage{setspace}
\usepackage{babel}
\usepackage{float}

\usepackage{xcolor}
\usepackage{subcaption}
\usepackage{comment}

\theoremstyle{plain}
\newtheorem{theorem}{Theorem}
\newtheorem*{theorem*}{Theorem}
\newtheorem{lemma}{Lemma}
\newtheorem{corollary}{Corollary}
\newtheorem*{corollary*}{Corollary}
\newtheorem{definition}{Definition}
\newtheorem*{remark*}{Remark}

\newcommand{\ds}{\displaystyle}

\def\R{{\mathbb R}}

\def\const{\textrm{const}}
\def\k{{\kappa}}

\def\R~{\widetilde{R}}

\title{Caustics of light rays and Euler's angle of inclination}
\author{Sergiy Koshkin* and Ivan Rocha\\
\\
*Corresponding author\\
Department of Mathematics and Statistics\\
University of Houston-Downtown\\
One Main Street\\
Houston, TX 77002\\
e-mail: koshkins@uhd.edu}

\date{}

\begin{document}

\maketitle

\begin{abstract}
Euler used intrinsic equations expressing the radius of curvature as a function of the angle of inclination to find curves similar to their evolutes. We interpret the evolute of a plane curve optically, as the caustic (envelope) of light rays normal to it, and study the Euler's problem for general caustics. The resulting curves are characterized when the rays are at a constant angle to the curve, generalizing the case of evolutes. Aside from analogs of classical solutions we encounter some new types of curves. We also consider caustics of parallel rays reflected by a curved mirror, where Euler's problem leads to a novel pantograph equation, and describe its analytic solutions. 
\bigskip

\bigskip


\end{abstract}

\section*{Introduction}

Caustic, from Latin ``burn", is a curve or surface that is particularly brightly lit. In geometric optics this happens because all light rays are tangent to it, and hence the light concentrates near it. A special class of caustics,  caustics by reflection or catacaustics (from Greek {\it catoptron}, mirror), were introduced by Tschirnhaus in 1682 \cite{KK,SS}. For them the rays come from a point source reflected in a curved mirror. Determining catacaustics was an important test case for the early calculus methods. Jacob Bernoulli studied them around 1692, and L'Hopital's calculus textbook (1696) included two chapters on them \cite{Boy46}. 

As Bernoulli and l'Hopital already realized, evolutes of plane curves, the sets of their centers of curvature, can be interpreted as a special case of caustics, when the light rays are normal to the curve. They already appear implicitly in Apollonius's {\it Conica} (c.\,200 BC), but the concept and the term ``evolute" were introduced by Huygens in his {\it Horologium Oscillatorum} (1673), in the course of designing an ideal pendulum clock. Other classical examples of caustics include caustics by refraction and tractrices.

In this paper we investigate caustics with a special property that attracted much attention early on. Huygens knew that cycloids are congruent to their own evolutes, and his pendulum construction relied on that fact \cite{SS}. Bernoulli showed in 1692 that logarithmic spirals have the same property. Half a century later Euler devoted two papers to studying plane curves that are similar (homothetic) to their evolutes \cite{Eu50,Eu66}. Aside from logarithmic spirals and cycloids, their circular relatives, epicycloids and hypocycloids, turned out to have this property as well. Many prominent French mathematicians, Gergonne \cite{Gerg}, Poisson \cite{Pois}, Binet \cite{Bin}, Puiseux \cite{Puis}, weighed in on the problem of finding more such curves. Puiseux gave a more or less complete analytic solution in 1844, which entered Salmon's 1852 textbook \cite[\textsection\textsection\,302-305]{Salm}, but was largely forgotten afterwards.

Prior to his work on evolutes, Euler introduced what we now call natural equations, and effectively derived the Frenet-Serret (moving frame) equations for plane curves \cite{Eu41}. But that was not quite what finding curves with similar evolutes required. Using parametric, or even natural, equations leads to a non-linear problem. Euler discovered, and others rediscovered, how to linearize it -- instead of the arclength one has to use the curve's angle to a fixed line as the parameter, the {\it angle of inclination}. Even with this insight, the general problem leads to delay differential equations, in addition to ordinary ones, whose analytic solutions were first found by Puiseux.

In this paper we generalize the Euler's problem to caustics specified by a direction field along the curve ({\it conormal field}), and show that the use of the angle of inclination linearizes it as well. Two special cases are studied in more detail. First, caustics of light rays that are tilted to the normal direction of the curve by a constant angle, we call them {\it skew-evolutes} of the curve. And second, caustics by reflection of a pencil of parallel rays in a curved mirror. For the skew-evolutes the problem leads to ordinary and delay differential equations analogous to the classical ones for evolutes, but some new types of curves appear as solutions. 

For caustics by reflection we encounter a more complex type of functional differential equation even in the simplest cases, a {\it pantograph equation}. Such equations, even with constant coefficients, were only studied since 1970s \cite{Fox, M-PN}. Ours is singular with variable coefficients, and, to the best of our knowledge, there is no general theory  of such equations to date. We investigate its analytic solutions, and find that, with the exception of the classical cycloidal mirror \cite{McL}, the mirrors they represent are not physically feasible. They bend over themselves, and obstruct incoming rays from reaching their inner parts. 

The paper is organized as follows. In Preliminaries we introduce the standard notation and terminology from geometry of plane curves, with some modifications, and in Section \ref{Caustic} we derive the equation of the general caustic generated by a conormal field along a curve. In Section \ref{Skew} we specialize to families of rays that make a constant angle with the curve and their caustics, the skew-evolutes. Analytic curves similar to their skew-evolutes are surveyed more or less completely. In Section \ref{Mirrors} we introduce caustics by reflection of a parallel pencil of rays in a curved mirror, and discuss their geometric properties.  Section \ref{Panto} discusses solving the pantograph equation to which the Euler's problem reduces, and interpreting the solutions geometrically.

\section{Preliminaries}\label{Prelim}

Here we recall the standard notation and terminology of differential geometry \cite{Gibs,Str}, but with some modifications convenient for our purposes. The position vector of a plane curve is denoted $r$, and its natural parameter, specified up to a choice of initial point and direction, $s$. Dots will denote derivatives with respect to $s$. Then $\dot{r}(s)$ is the unit velocity vector to the curve. 

The curves we consider have smooth (infinitely differentiable) parametrizations, but will be allowed to have cusps, that is points where the velocity reverses direction. Such curves are often called fronts, or wavefronts \cite{FT}, but we will not use this terminology. 
\begin{definition} Let $T$ denote the unit tangent that varies smoothly along the curve except for reversing the direction at the cusps. The unit normal $N$ is obtained by rotating $T$ counterclockwise by $90^\circ$, and also reverses direction at the cusps. The curvature $\k$ is taken with a sign that changes at the cusps.
\end{definition}
This convention is different from the standard one, but it is natural for the parametrizations that we will use, and allows analytic expressions for curves we consider to remain valid across cusps. The vectors $T$ and $N$ form a {\it moving frame} along the curve, and $\dot{T}=\k N$ always holds. Together with the equation for $N$ we obtain the Frenet-Serret equations, which for plane curves were derived already by Euler \cite{Eu41}:
\begin{equation}\label{MovFram}
    \begin{cases}\dot{T}=\k N\\
    \dot{N}=-\k T.\end{cases}
\end{equation}
\begin{definition}\label{Evo} The reciprocal of the curvature $R=\frac1{\k}$ is called the {\bf (signed) radius of curvature}, and the point $c=r+RN$ is called the center of curvature. The {\bf evolute} of a curve is the locus of its centers of curvature. Quantities associated with evolutes (and later caustics) will be labeled by index $1$, and we adopt the convention that $ds_1$ and $dR$ have the same sign. \end{definition}
\noindent A natural equation of the curve $\k=\k(s)$, and its equivalent $R=R(s)$, sometimes called Ces\'aro equation \cite{Law}, are (almost) intrinsic, i.e. depend on its geometric shape rather than on parametrization or position in the plane, up to the ambiguity in the choice of natural parameter. Natural equations of many known plane curves can be found in \cite{Law} and \cite{Yat}.

A natural parameter of the curve may no longer be natural for its evolute. We compute
\begin{equation}\label{cdot}
\dot{c} = \frac{d}{ds}(r+RN)=\dot{r}+\dot{R}\,N+R\dot{N} = T + \dot{R}\,N-R\,\k\!\,T = \dot{R}\,N.
\end{equation}
Since $\dot{c}$ is the velocity of moving along $c$, and $N$ is a unit vector, equation \eqref{cdot} implies that $\frac{ds_1}{ds}=\dot{R}=\frac{dR}{ds}$, i.e. the change of the arclength along the evolute is equal to the corresponding change in the radius of curvature along the source curve. 

Euler's idea was to use instead of a natural parameter the angle $\theta$ of the curve with a fixed line, which we might as well identify with the $x$-axis \cite{Str}, Figure \ref{Inclin}(b).
\begin{definition} The angle $\theta$ between the tangent to a plane curve and the $x$-axis is called its angle of inclination at that point. The equation $R=R(\theta)$, where $R$ is the (signed) radius of curvature, will be called the {\bf inclination equation} of the curve.
\end{definition}
\noindent Similar equations were used by Whewell \cite{Whew}, but he worked with $s(\theta)$ rather than $R(\theta)$. 

Thinking of curving as an infinitesimal rotation we see that $\frac{ds}{R}=d\theta$ when $R$ is positive. To make this hold universally, we need to change the meaning of $s$ as the arclength since we allow negative $R$. 

\begin{definition} We redefine $s$ as the parameter satisfying $\frac{ds}{d\theta}=R(\theta)$, and continue to use dots for derivatives with respect to this $s$. 
\end{definition}
\noindent Our $s$ can increase as well as decrease with $\theta$ since $R$ can be negative, but $\theta$ will always strictly increase along a curve, and \eqref{MovFram} continue to hold. We also have $\frac{ds_1}{R_1}=d\theta$, since the change of the angle is the same along the curve and its evolute, so 
\begin{equation}\label{EvoRad}
R_1=R\frac{dR}{ds}=\frac{dR}{d\theta}.
\end{equation}
Now we can see why the angle of inclination might be a better parameter -- while the relation of $R_1(s)$ to $R(s)$ is non-linear, $R_1(\theta)$ is just the derivative of $R(\theta)$. From now on we will denote derivatives with respect to $\theta$ by $'$. Note that $R_1(\theta)$ is not an inclination equation of the evolute, rather it is $R_1(\theta_1)=R'(\theta_1-\frac{\pi}2)$, because for evolutes $\theta_1=\theta+\frac{\pi}2$. Still, in terms of inclination equations, passing to the evolute amounts to just taking the derivative and shifting the argument, a linear operation. 
\begin{figure}[!ht]
\vspace{-0.1in}
\begin{centering}
(a)\ \ \ \includegraphics[width=1.1in]{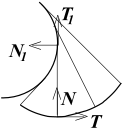} \hspace{1.0in} 
(b)\ \ \ \includegraphics[width=1.5in]{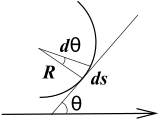}
\par\end{centering}
\vspace{-0.2in}
\hspace*{-0.1in}\caption{\label{Inclin} (a) Moving frames for a curve and its evolute; (b) angle of inclination, arclength and the radius of curvature for a plane curve.}
\end{figure}

Moreover, inclination equations are no less intrinsic than natural equations. The only extrinsic choices are the fixed line and the direction of measuring the angle, just as with a natural parameter we had to choose the starting point and the direction. From an inclination equation we can easily retrieve a natural parameter, $s=\int R(\theta)d\theta$, and even parametric equations of the curve \cite{Str}:
\begin{equation}\label{Inc-Par}
x(\theta)=\int R(\theta)\cos\theta\,d\theta;\ \ \ \ \ \ \
y(\theta)=\int R(\theta)\sin\theta\,d\theta\,.
\end{equation}
Note that under our sign conventions the unit tangent to the curve is always $T(\theta)=(\cos\theta,\sin\theta)$, and it reverses the direction relative to the direction of increasing $\theta$ whenever $R(\theta)$ changes sign at a cusp \cite[7.5]{Gibs}. The unit normal is $N(\theta)=(-\sin\theta,\cos\theta)$, and it also reverses direction at cusps.

\section{Conormal fields and their caustics}\label{Caustic}

To generalize Euler's observation, let us first interpret evolutes optically. Picture a windowless room with evenly spaced lasers mounted on the walls near the ceiling. Some places on the ceiling will be particularly brightly lit. Those are the places where multiple laser rays ``lump" together. They are typically clustered around a curve determined by the shape of the room, it is the curve that all the rays are tangent to. Such a curve is defined for any smooth one-parameter family of smooth curves and is known as its {\it envelope}. In the case of light rays, the envelope is called the {\it caustic}. 

To find the caustic, note that there are two ways of moving within a family of plane curves. We can move along a curve in the family, or we can move transversally to it, from one curve to another. Typically, these two directions of motion are distinct, but when moving along the envelope they merge into a single one, the direction of its tangent. If the lasers are normal to the wall shaped by a naturally parametrized curve $r(s)$, then the family of light rays is given by $r_{t}(s)=r(s)+t N(s)$. The two directions are $\frac{\partial r_{t}}{\partial s}=\dot{r_{t}}$ and $\frac{\partial r_{t}}{\partial t}=N$, respectively. On the envelope these two vectors must be collinear, i.e. their cross-product must equal to $0$. By the moving frame equations \eqref{MovFram},
\[\label{NormEnv}
\frac{\partial r_{t}}{\partial s}\times\frac{\partial r_{t}}{\partial t}=\dot{r_{t}}\times N=(\dot{r}+t \dot{N})\times N
=(1-t \k )\,T\times N\,.
\]
Since $T$ and $N$ are orthogonal unit vectors, this product is $0$ if and only if $t=\frac{1}{\k}$. Therefore, the caustic $c=r+\frac{1}{\k} N$ is exactly the evolute of $r$. 
\begin{figure}[!ht]
\vspace{-0.1in}
\begin{centering}
(a)\ \ \ \includegraphics[width=0.8in]{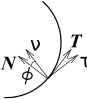} \hspace{0.8in} 
(b)\ \ \ \includegraphics[width=0.6in]{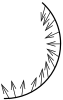} 
\par\end{centering}
\vspace{-0.2in}
\hspace*{-0.1in}\caption{\label{CoGeom} (a) Moving coframe tilted by angle $\phi$; (b) conormal field along a curve.}
\end{figure}

Now imagine that the lasers are not normal to the wall, but point at varying directions, as on Figure \ref{CoGeom}(b). The next definition introduces the corresponding generalization of evolutes.
\begin{definition}\label{coframe} Denote the unit vector along the direction of the ray $\nu$, and call the angle $\phi$ between $\nu$ and $N$ the {\bf tilting angle}, with the positive tilt being in the direction of $T$. Denote $\tau$ the unit vector perpendicular to $\nu$ such that $\tau, \nu$ form a right pair. The fields $\nu(s),\tau(s)$ are called {\bf conormal and cotangent fields}, respectively, and $\tau, \nu$ the moving coframe. The caustic generated by rays directed along a field $\nu$ is called the {\bf conormal caustic}.
\end{definition}
It turns out that there are straightforward analogs of Frenet-Serret equations \eqref{MovFram} for the moving coframe. Differentiating the relations $\tau \cdot \tau =1$, $\tau \cdot \nu=0$, and $\nu \cdot \nu = 1$, we obtain for some function $\chi(s)$:
\begin{equation}\label{Coframe}
    \begin{cases}\dot{\tau}=\chi\,\nu\\
    \dot{\nu}=-\chi\,\tau.\end{cases}
\end{equation}
Clearly, $\chi(s)$ plays a role analogous to curvature, but reflects not only curving of the curve itself but also variation in the tilting angle along it.
We can specify a conormal field by a field of tilting angles $\phi(s)$, or, more conveniently for our purposes, $\phi(\theta)$. Using the moving coframe we will generalize Euler's formula for the inclination equation of the caustic in these terms. Recall that dots stand for derivatives with respect to $s$ and primes for derivatives with respect to $\theta$.
\begin{theorem}\label{CaustInc} Let $R$ be the (signed) radius of curvature of a plane curve with rays at each point tilted at angle $\phi$ to its normal, and assume that $\phi'\neq1$ at any point. Then the angle of inclination of the caustic at the point of tangency with the ray is $\theta_1=\theta+\frac{\pi}2-\phi$, and its (signed) radius of curvature is
\begin{equation}\label{R1theta}
R_1=\frac{1}{(1-\phi')^2}\left(\left((1-2\phi')\,\sin\phi+\frac{\phi''}{1-\phi'}\,\cos\phi\right)R+\cos\phi\,R'\right).
\end{equation}
\end{theorem}
\begin{proof}
From Definition \ref{coframe} we have:
\begin{equation}\label{TiltFram}
    \begin{cases}\tau = \cos\phi\,T - \sin\phi N\\
    \nu = \sin\phi\,T+\cos\phi\,N.\end{cases}
\end{equation}
Differentiating the first equation, and taking into account \eqref{MovFram}, \eqref{Coframe} and \eqref{TiltFram}, we find:
\begin{equation}\label{TiltCurv}
\chi=\k - \dot{\phi}=\frac{1-\phi'}R.
\end{equation}
Using the moving coframe equations we can find parametric equation of the caustic and its velocity analogously to the normal case: \begin{align}\label{CoCaust}
c&=r+\frac{\cos\phi}{\chi}\,\nu\,;\\
\dot{c}&=\dot{r}-\frac{\dot{\chi}}{\chi^2}\,\cos\phi\,\nu-\frac{\dot{\phi}}{\chi}\,\sin\phi\,\nu+\frac{1}{\chi}\,\cos\phi\,\dot{\nu} =\Big(\big(1-\frac{\dot{\phi}}{\chi}\big)\,\sin\phi-\frac{\dot{\chi}}{\chi^2}\,\cos\phi\Big)\,\nu\,.\notag
\end{align}
The coefficient $\dot{c}\cdot\nu$ of $\nu$ gives the factor relating the natural parameters of the two curves: $ds_1=(\dot{c}\cdot\nu)ds$. As expected, conormal rays are tangent to the caustic. This implies that the angles of inclination along the curve and its caustic are related as $\theta_1=\theta+\frac{\pi}2-\phi$, Figure \ref{IncShift}(a). Since $ds=R\,d\theta$ and $ds_1=R_1\,d\theta_1$:  
\begin{equation}\label{R1Short}
R_1=\frac{ds_1}{d\theta_1}=\frac{(\dot{c}\cdot\nu)\,ds}{d\theta-d\phi}=\frac{(\dot{c}\cdot\nu)R\,d\theta}{d\theta-d\phi}=\frac{\dot{c}\cdot\nu}{1-\phi'}\,R\,.
\end{equation}
Combining \eqref{CoCaust} and \eqref{R1Short} we compute:
\begin{equation}\label{R1Gen}
R_1=\frac{1}{1-\phi'}\left(\left(1-\frac{\dot{\phi}}{\chi}\right)\,\sin\phi-\frac{\dot{\chi}}{\chi^2}\,\cos\phi\right)\,R
\end{equation}
The formula \eqref{R1theta} now follows from \eqref{TiltCurv} and \eqref{R1Gen}.
\end{proof}
\noindent If $\phi'=1$ at a point then the caustic flattens at the corresponding point, $\k_1=0$. If $\phi'=1$ over an interval then $\theta_1=\const$ on that interval, i.e. the rays are parallel to each other on it and the caustic stays ``at infinity". For the evolutes, $\phi=\phi'=\phi''=0$, so this reduces to the familiar $R_1=R'$. Moreover, as long as the tilting angle is known as a function of inclination, finding the caustic is always a linear problem.
\begin{figure}[!ht]
\vspace{-0.1in}
\begin{centering}
(a)\ \ \ \includegraphics[width=1.6in]{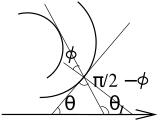}
\hspace{0.3in}
(b)\ \ \ \includegraphics[width=1.1in]{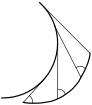}
\par\end{centering}
\vspace{-0.2in}
\hspace*{-0.1in}\caption{\label{IncShift} (a) angles of inclination along a curve and its caustic; (b) skew-evolute.}
\end{figure}
Theorem \ref{CaustInc} allows us to express similarity of a curve to its caustic as an equation for $R(\theta)$. One may be tempted to equate the right hand side of \eqref{R1theta} to $aR(\theta)$, where $a$ is the similarity factor, but that overlooks the change in the angle of inclination and the residual ambiguities in the inclination equations. The angle of inclination of the caustic is shifted by $\frac{\pi}2-\phi(\theta)$ relative to the angle of inclination of the curve. If we rotate the curve $\theta$ will shift by the angle of rotation, and if we change the direction of measuring it, it will switch sign. Thus, the general condition we want is $R_1\left(\theta_1\right)=aR(\pm(\theta_1-\beta))$, where $\theta_1$ is the caustic's angle of inclination. But \eqref{R1theta} gives us $R_1$ as a function of $\theta$, not $\theta_1$.
\begin{corollary}\label{SimCond}
Let $R=R(\theta)$ be the inclination equation of a plane curve with a conormal caustic defined by the field of tilting angles $\phi(\theta)$ such that $\phi'\neq1$. Then the curve and the caustic are similar if and only if $R$ satisfies the functional differential equation:
\begin{equation}\label{SimEq}
\frac{1}{(1-\phi')^2}\left(\left((1-2\phi')\,\sin\phi+\frac{\phi''}{1-\phi'}\,\cos\phi\right)R+\cos\phi\,R'\right)=aR\left(\pm(\theta+\frac{\pi}2-\phi-\beta)\right)
\end{equation}
for some real $a$ and $\beta$.
\end{corollary}
\noindent For evolutes the similarity equation \eqref{SimEq} reduces to $R'(\theta)=aR(\pm(\theta-\alpha))$, where $\alpha=\beta-\frac\pi2$.

\section{Curves similar to their skew-evolutes}\label{Skew}

As we mentioned, one can think of evolutes as conormal caustics of families of rays with zero tilting angle at each point. The next simplest case is to consider families with a fixed angle $\phi$ other than zero, see Figure \ref{IncShift}(b).

\begin{definition}\label{Skew-ev} The {\bf skew-evolute} of a plane curve with tilting angle $\phi$ is the caustic of rays that make a constant angle $\frac\pi2-\phi$ with the curve.
\end{definition}
\noindent For skew-evolutes we still have $\phi'=\phi''=0$, so \eqref{SimEq} simplifies to
\begin{equation}\label{SkewSimi}
\cos\phi\,R'(\theta)+\sin\phi\,R(\theta)=aR(\pm(\theta-\alpha)),
\end{equation}
where $\alpha=\beta-\frac\pi2$.
For the purposes of solving this equation it is convenient to distinguish three types of cases \cite{Salm}. 
\begin{definition}\label{SkewCases}We say that a curve and its skew-evolute are similar {\bf point by point} when the right hand side in \eqref{SkewSimi} is just $aR(\theta)$, similar {\bf in inverse position} when it is $aR(\alpha-\theta)$, and similar {\bf with delay/advance} when it is $aR(\theta-\alpha)$.
\end{definition}
\noindent For the inverse position case the direction of motion along the skew-evolute is reversed compared to the direction on the curve. For the delay/advance case, if one thinks of $\theta$ as ``time" \eqref{SkewSimi} means that the skew-evolute's radius curvature at any ``moment" scales the curve's radius at some other ``time" $\theta-\alpha$. When $\alpha>0$ we have a delay, and when $\alpha<0$ we have an advance. An advance can be converted into a delay by switching from $\theta$ to $-\theta$ as the angle of inclination, and changing the  signs of $\phi$ and $a$, so it suffices to consider only the case $\alpha>0$. 

For (normal) evolutes all curves similar to them point by point and in inverse position were already found by Euler in \cite{Eu50} and \cite{Eu66}, and some were known to Huygens and Bernoulli even earlier. Finding them reduces to solving simple ordinary differential equations (ODE). But the delay case requires solving a different type of equation, a delay differential equation (DDE) \cite{Falb}. Unlike the ODE, whose solutions depend on a few free constants, general solutions to DDE contain infinitely many of them. Hence, this case produces the ``most" curves. A theory of DDE was only developed after Euler, and the curves similar to their evolutes with delay were characterized by Puiseux a century after him \cite{Puis}. Let us see what happens when we allow constant tilting angles that are non-zero.
\begin{figure}[!ht]
\vspace{-0.1in}
\begin{centering}
(a)\includegraphics[width=32mm,height=40mm]{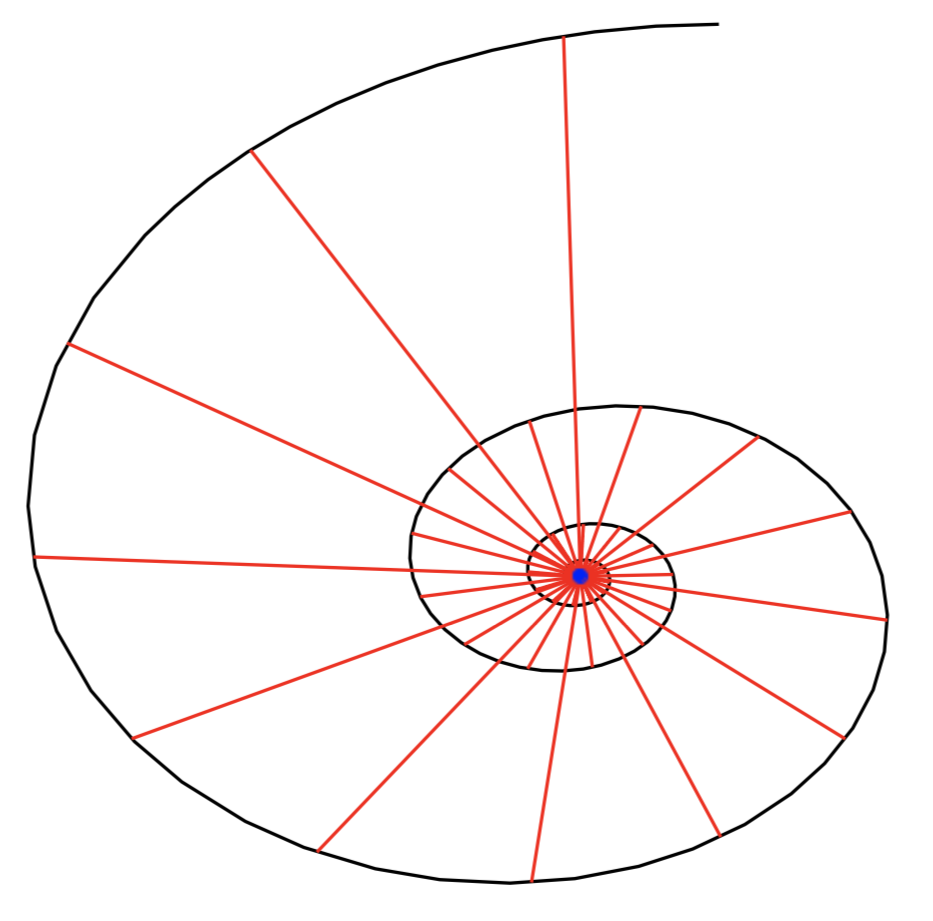} \hspace{0.2in} (b) \includegraphics[width=40mm,height=40mm]{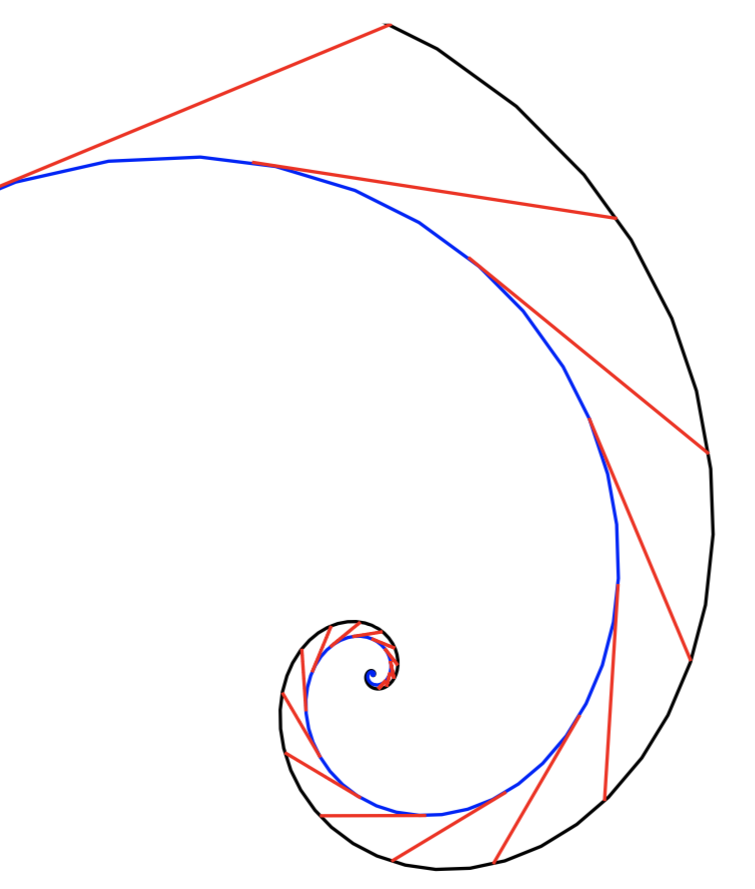} \hspace{0.2in} (c)\includegraphics[width=40mm,height=40mm]{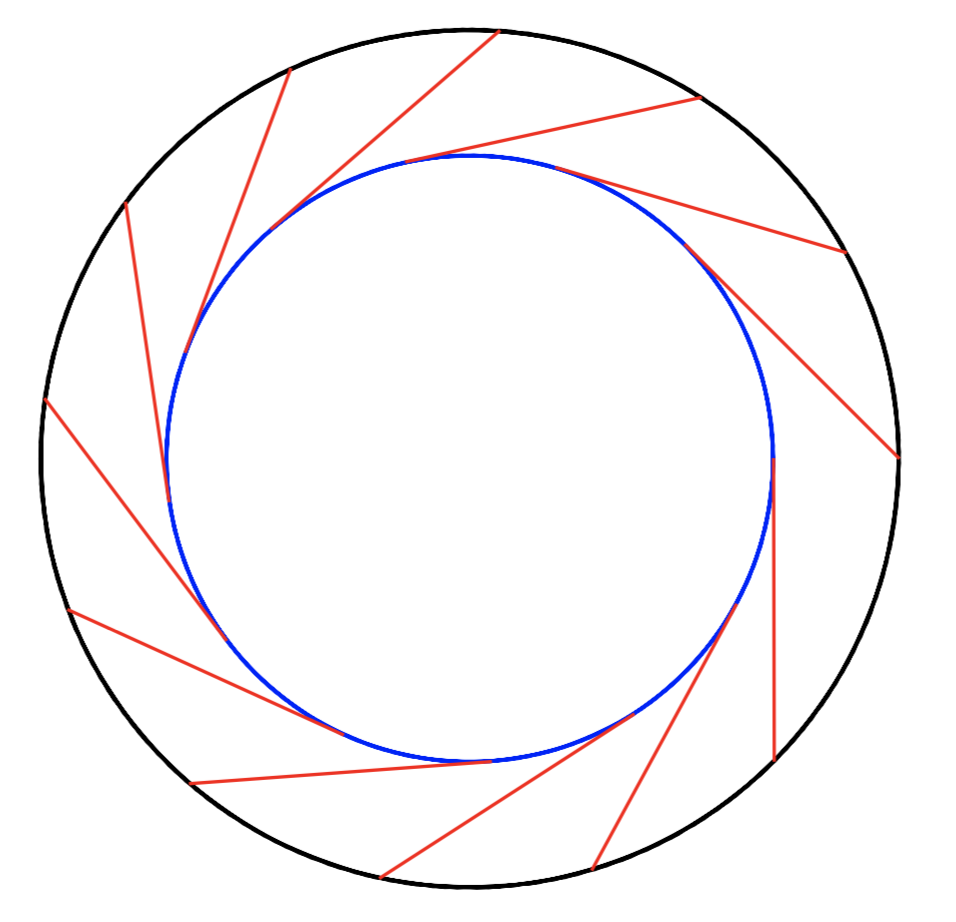}
\par\end{centering}
\vspace{-0.2in}
 \hspace*{-0.1in}\caption{\label{SkewPt} (a) Logarithmic spiral whose skew-evolute is a single point, $\phi=85^\circ$; (b) logarithmic spiral whose skew-evolute is a similar spiral, $\phi=45^\circ$; (c) skew-evolute of a circle is another circle.}
\end{figure}
\begin{figure}[!ht]
\vspace{-0.1in}
\begin{centering}
(a)\includegraphics[width=40mm]{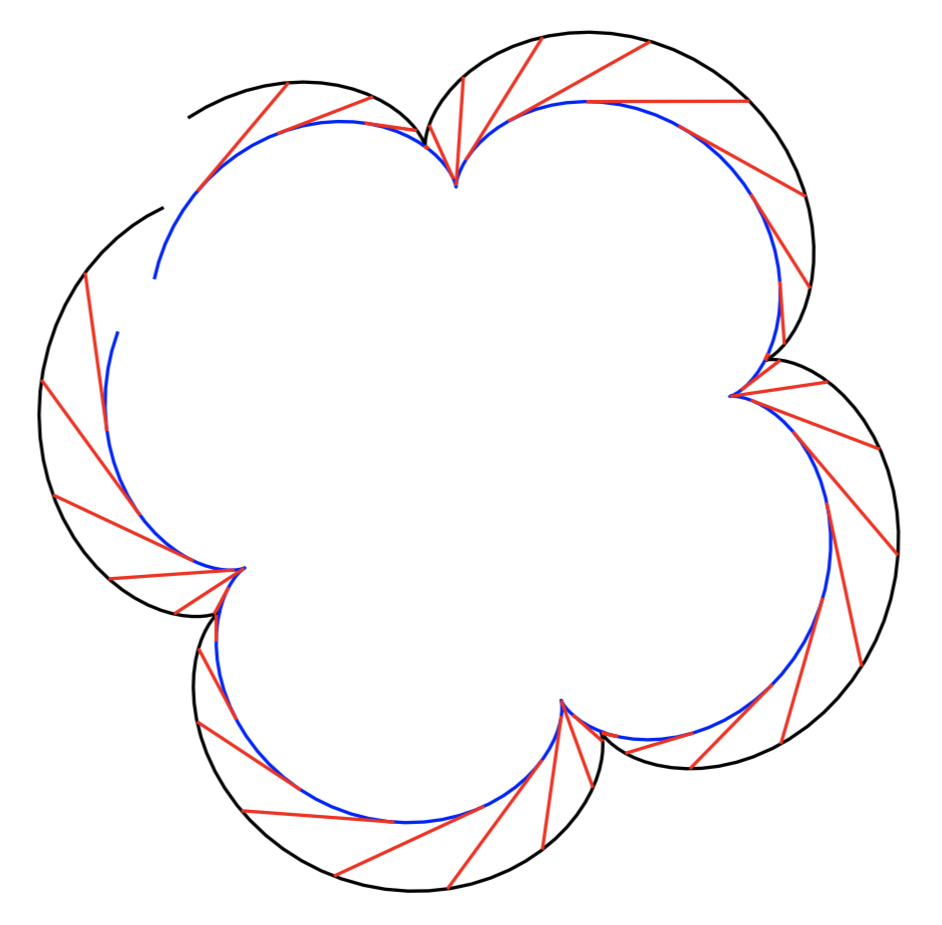} \hspace{0.2in} 
(b) \includegraphics[width=40mm]{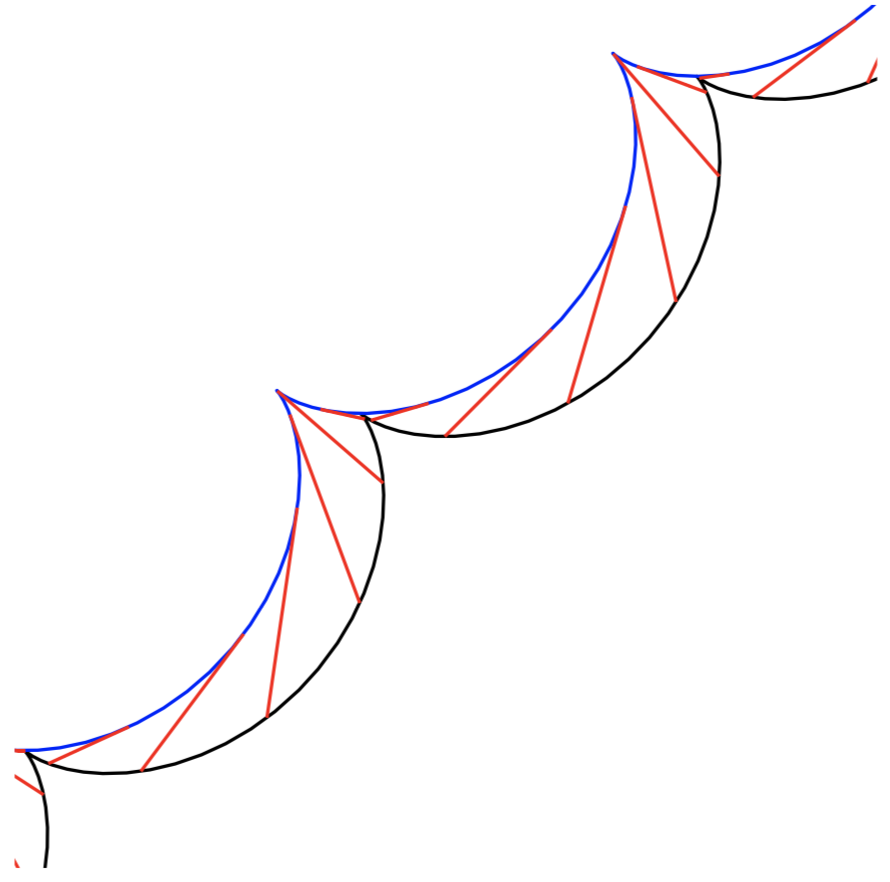} \hspace{0.2in} 
(c)\includegraphics[width=40mm]{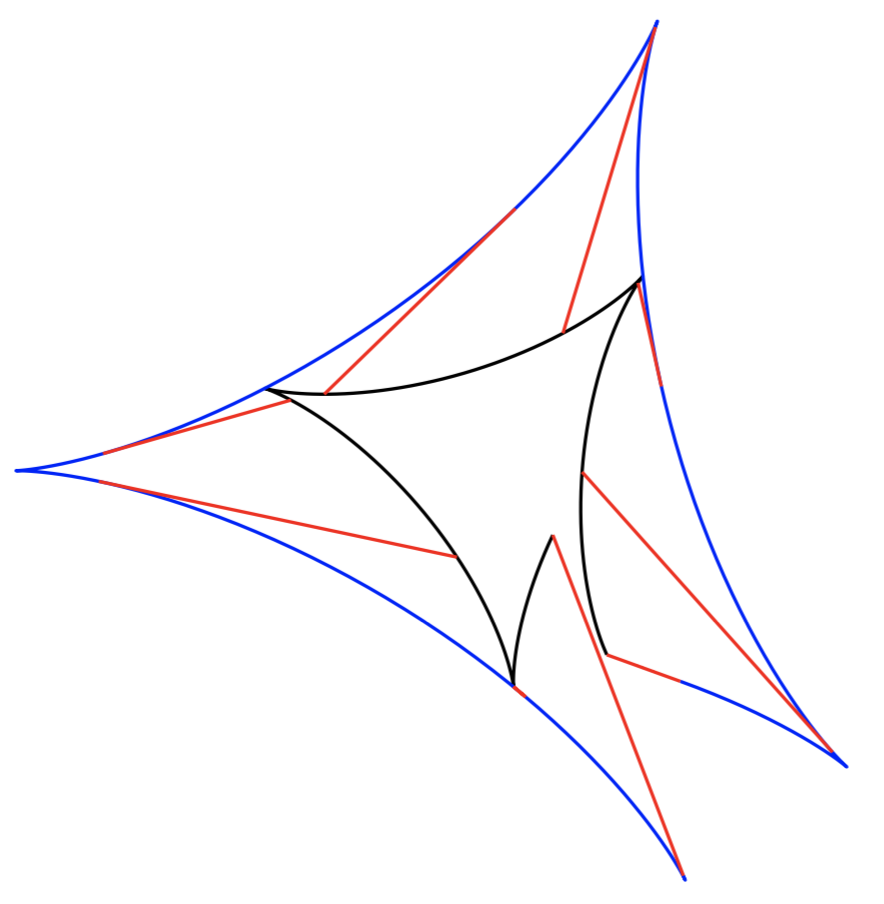}
\par\end{centering}
\vspace{-0.2in}
 \hspace*{-0.1in}\caption{\label{SkewCyc} Skew deformations of cycloidal curves with $\phi=\frac{\pi}{4}$: (a) epicycloid, $a=\frac{3}{4}$; (b) cycloid, $a=1$; (c) hypocycloid, $a=2$.}
\end{figure}

\begin{corollary}\label{Skew12} 
\textup{(i)} Plane curves similar to their evolutes point by point are logarithmic spirals with inclination equations $R(\theta)=Ae^{\frac{a-\sin\phi}{\cos\phi}\theta}$ that degenerate into circles when $a=\sin\phi$, see Figure \ref{SkewPt}.

\textup{(ii)} Plane curves similar to their evolutes in inverse position have inclination equations that solve the second order ODE $R''(\theta)+\left(\frac{a^2}{\cos^2\phi}-\tan^2\phi\right)R=0$. They are deformations of epicycloids, cycloids and hypocycloids, when the expression in parentheses is positive, circle's involutes when it is $0$, and have inclination equations $R(\theta)=A\,\cosh(b\theta)+B\,\sinh(b\theta)$ when it is negative, where $b^2=\tan^2\phi-\frac{a^2}{\cos^2\phi}$, see Figures \ref{SkewCyc},\ref{SkewHyp}.
\end{corollary}
\begin{proof} The case (i) is straightforward by solving the point by point equation $\cos\phi\,R'(\theta)+\sin\phi\,R(\theta)=aR(\theta)$. For (ii) \eqref{SkewSimi}
reduces to $\cos\phi\,R'(\theta)+\sin\phi\,R(\theta)=aR(\alpha-\theta)$. We can solve it by using a general trick for solving functional differential equations with idempotent functions, such as $\theta\mapsto\alpha-\theta$, in the argument \cite{Falb}. Dividing by $\cos\phi$ transforms it into $R'(\theta)=-\tan\phi\,R(\theta)+\frac{a}{\cos\phi}R(\alpha-\theta)$. Replacing $\theta$ by $\alpha-\theta$ in it we have $R'(\alpha-\theta)=-\tan\phi\,R(\alpha-\theta)+\frac{a}{\cos\phi}R(\theta)$. Now setting $Q(\theta):=R(\alpha-\theta)$ we obtain a first order system of ODE for $R$ and $Q$:
$$
\begin{cases}
R'=-\tan\phi\,R+\frac{a}{\cos\phi}Q \\
Q'=\tan\phi\,Q-\frac{a}{\cos\phi}R.
\end{cases}
$$
Differentiating the first equation and eliminating $Q'$ using the second we reduce it to the claimed second order ODE. It is the equation of the harmonic oscillator and its hyperbolic analog, whose solutions are well-known.
\end{proof}
\begin{figure}[!ht]
\vspace{-0.1in}
\begin{centering}
(a)\includegraphics[width=40mm]{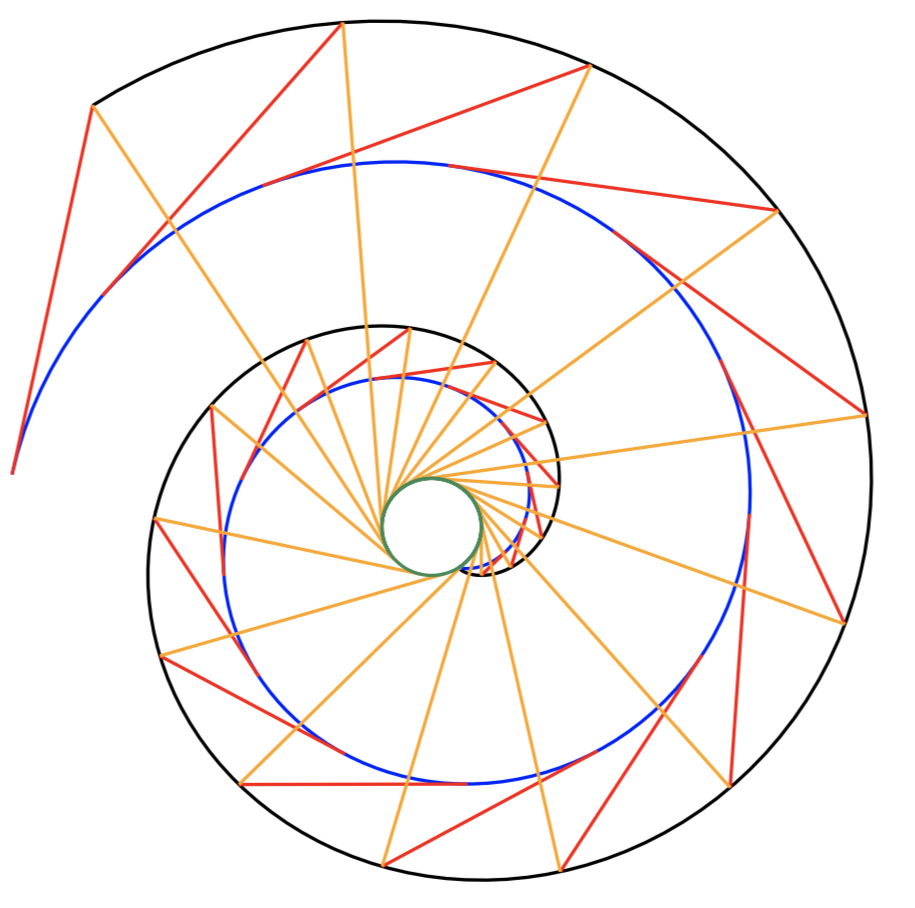}\hspace{0.2in} 
(b)\includegraphics[width=40mm]{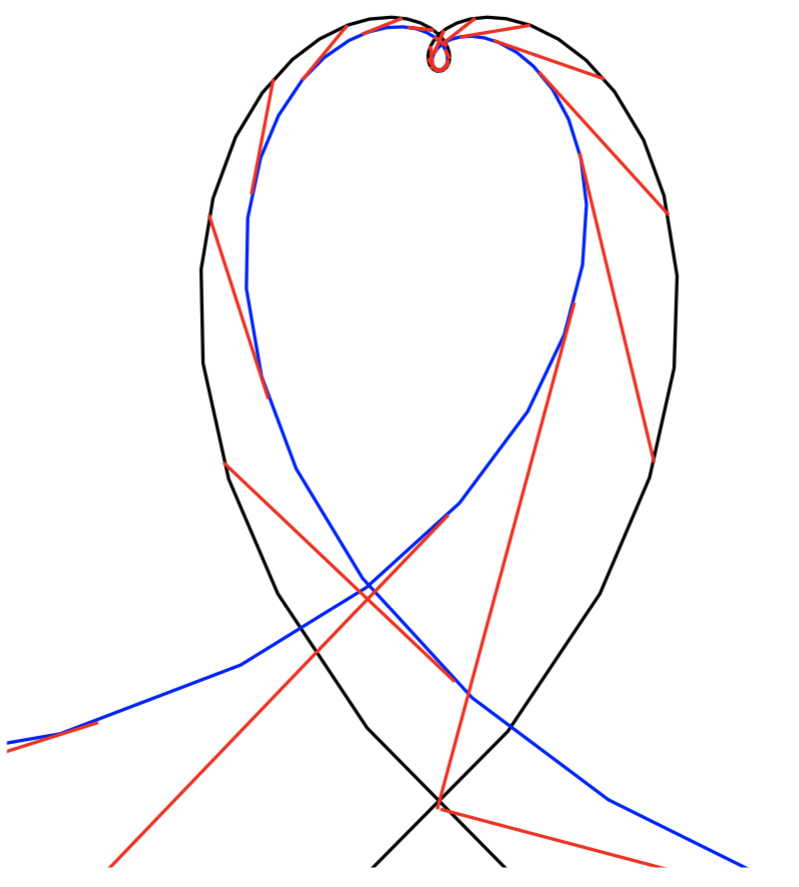} \hspace{0.2in} 
(c) \includegraphics[width=40mm]{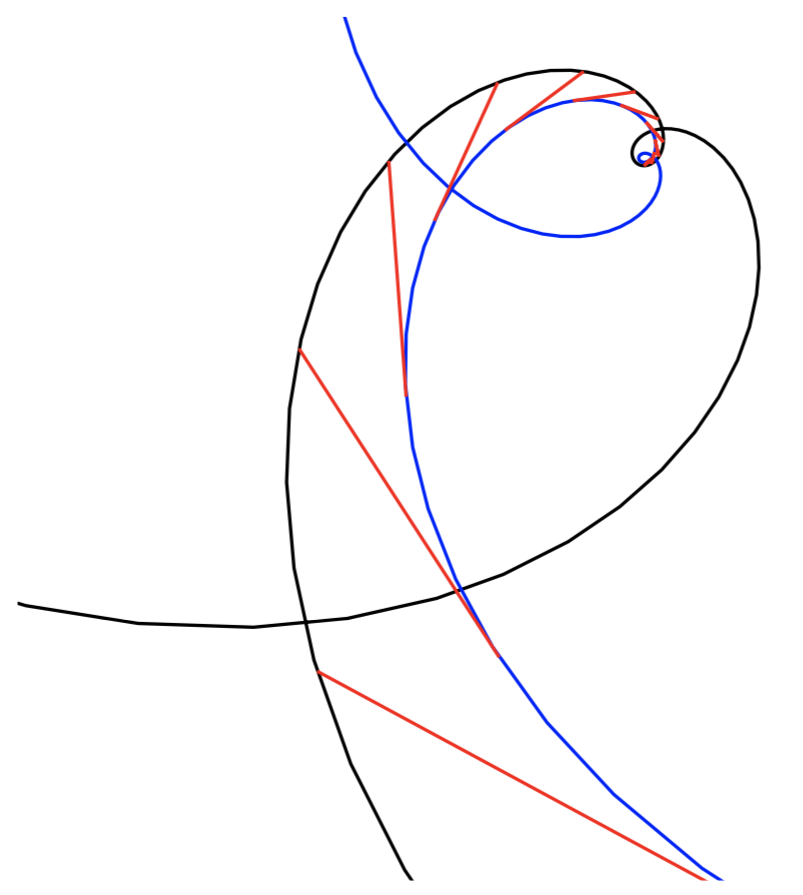} \hspace{0.2in}
\par\end{centering}
\vspace{-0.2in}
 \hspace*{-0.1in}\caption{\label{SkewHyp} (a) Skew-evolute of circle's involute; hyperbolic skew-cycloids with inclination equations: (b) $R(\theta)=\cosh(\theta)$; (c) $R(\theta)=2\cosh(\theta)-\sinh(\theta)$.}
\end{figure}
Note that the second order equations in (ii) have two constants of integration, one of which can be eliminated using that $Q(\theta)=R(\alpha-\theta)$. If both are chosen freely this choice will specify $\alpha$. Many of the resulting curves are direct analogs of classical curves similar to their (normal) evolutes, Figures \ref{SkewPt}, \ref{SkewCyc}. However, some new curves appear as well. When $a=\pm\sin\phi$ we get circle involutes (curves whose evolutes are circles), and when $a^2<\sin^2\phi$ we get unnamed curves with points of self-intersection, see Figure \ref{SkewHyp}.

The remaining equation  $R'(\theta)\cos\phi+R(\theta)\sin\phi=aR(\theta-\alpha)$ can not be reduced to an ODE by an analytic trick, but it is a linear delay differential equation (DDE) with constant coefficients. Some special solutions to it can be found explicitly by the method of characteristics, used already by Puiseux \cite{Puis} to find curves similar to their evolutes.

We begin by looking for solutions of the form $R(\theta)=e^{\lambda\theta}$. Substituting $R$ into the equation gives us the characteristic equation, which reduces to  $(\lambda+\tan\phi)\,e^{\alpha\,\lambda}=\frac{a}{\cos\phi}$. It is further reduced to a more standard form by making the substitution $z=\alpha(\lambda+\tan\phi)$, then
$$
ze^z=\frac{\alpha a}{\cos\phi}\,e^{\alpha\tan(\phi)}.
$$
This is known as the Lambert equation. Although it can not be solved in elementary functions, it is almost as well studied. The solutions are given by the {\it Lambert function} $W$, the inverse of $z e^{z}$, which is implemented in standard computer algebra systems \cite{CGHJK}. Like the complex logarithm $\textup{Ln}$, which is the inverse of $e^z$, it is multi-valued, and its branches $W_k$ are indexed by integers. Depending on the value of the right hand side this equation has two, one or no real solutions $\lambda_{0,1}$, and infinitely many complex conjugate ones. We can express the roots generally as
\begin{equation}\label{Lambk}
\lambda_k=\xi_k\pm i\eta_k=\frac1{\alpha}W_k\left(\frac{\alpha a}{\cos\phi}\,e^{\alpha\tan\phi}\right)-\tan\phi.
\end{equation}
Since our DDE is linear its general solution is the sum of particular ones for all possible $\lambda_k$. These sums are potentially infinite.
\begin{corollary}\label{Skew3} 
Plane curves similar to their evolutes with delay have inclination equations (formally) representable in the following form, where $\lambda_k$ are the solutions \eqref{Lambk} of the Lambert equation:
\begin{equation}\label{DelEvo}
    R(\theta)=A_0\,e^{\lambda_0\theta}+A_1\,e^{\lambda_1\theta}+\sum\limits_{k=2}^{\infty} e^{\xi_k\theta}(A_k\,\cos(\eta_k\theta)+B_k\,\sin(\eta_k\theta)).
\end{equation}
\end{corollary}
\begin{figure}[!ht]
\vspace{-0.1in}
\begin{centering}
(a)\includegraphics[width=70mm,height=36mm]{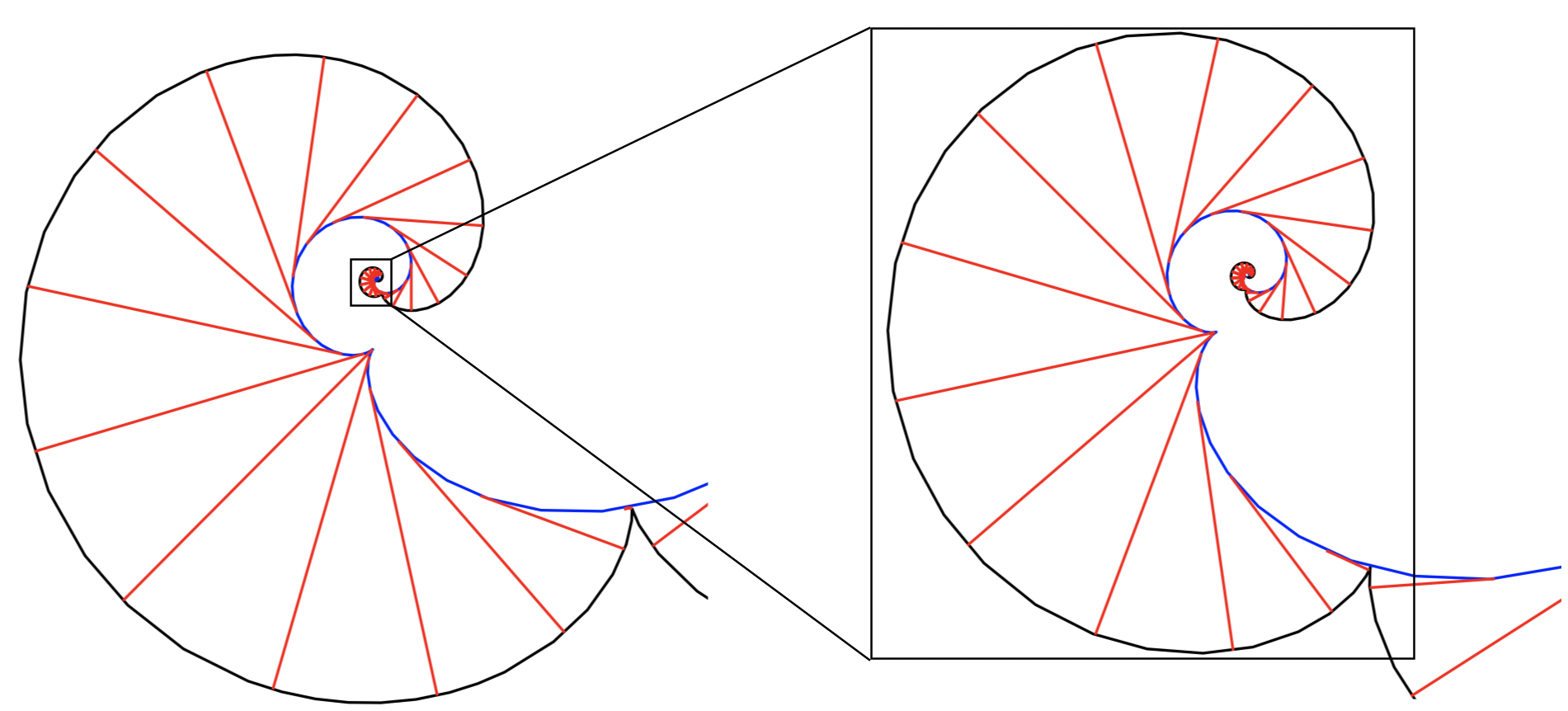}
\hspace{0.1in}
(b)\includegraphics[width=28mm,height=36mm]{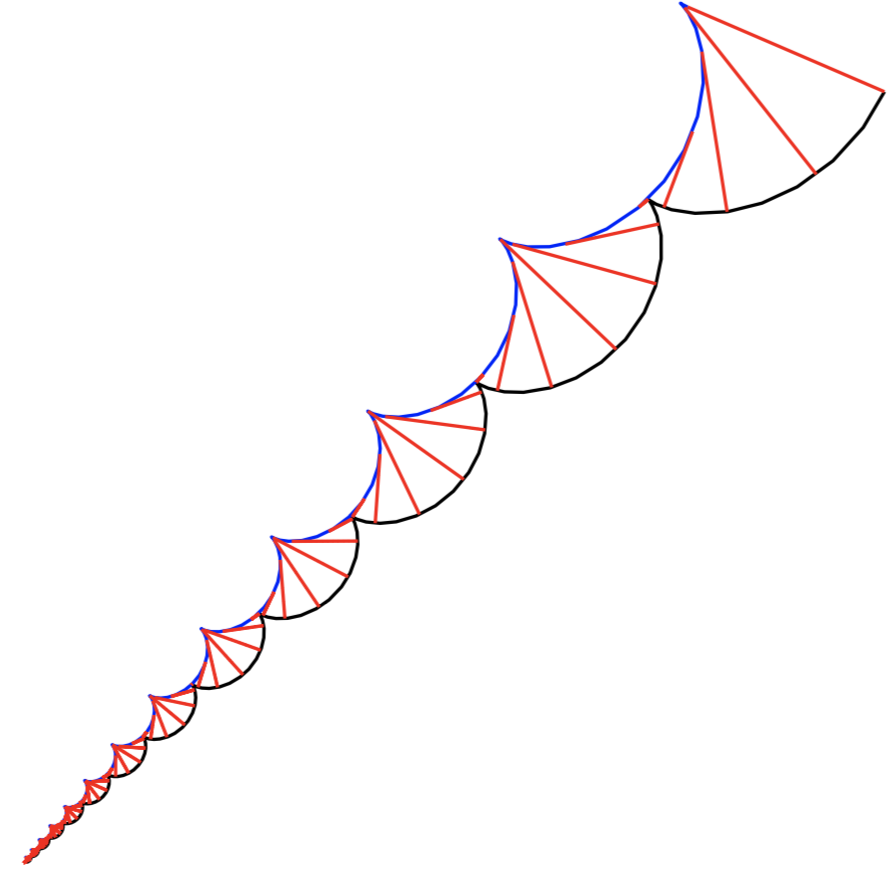}
\hspace{0.1in}
(c)\includegraphics[width=32mm, height=36mm]{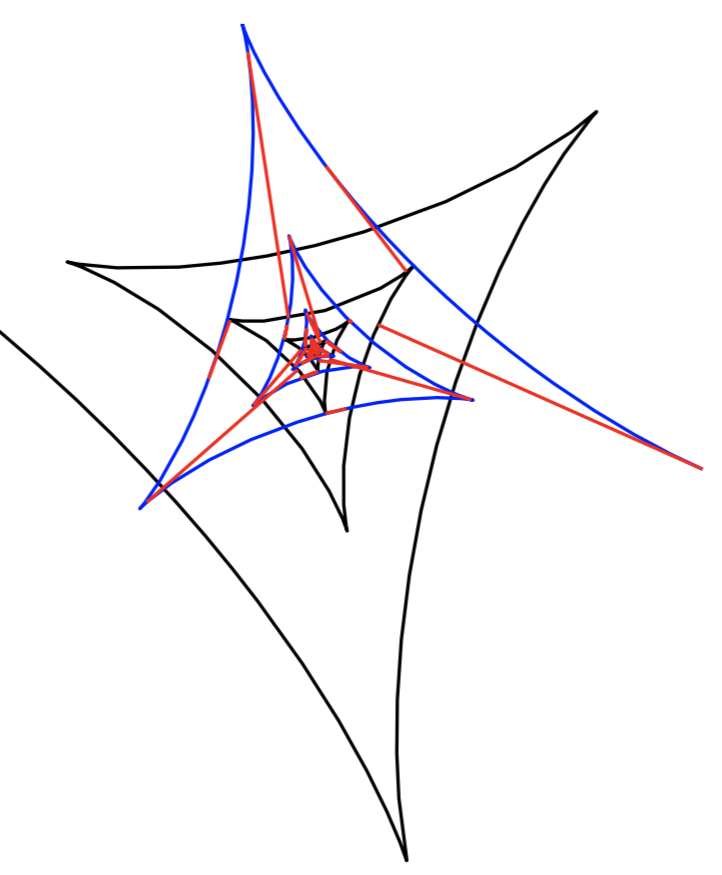}
\par\end{centering}
\vspace{-0.2in}
 \hspace*{-0.1in}\caption{\label{CuspSpir} (a) Puiseux epicycloidal spiral, the inset shows self-similarity; (b) Puiseux cycloid; (c) Puiseux hypocycloidal spiral.}
\end{figure}
\begin{figure}[!ht]
\vspace{-0.1in}
\begin{centering}
(a) \includegraphics[width=33mm]{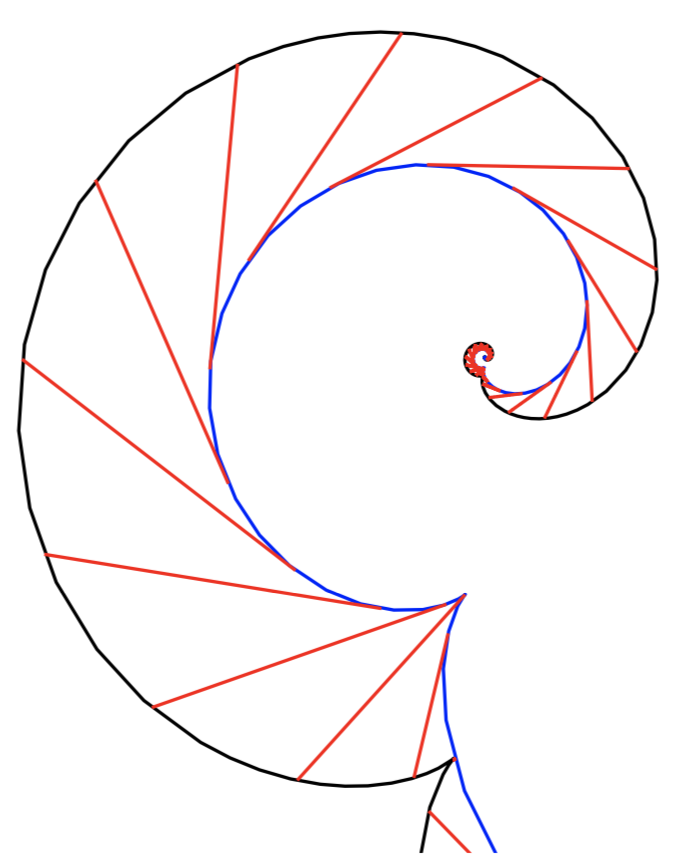}
\hspace{0.5in}
(b) \includegraphics[width=33mm]{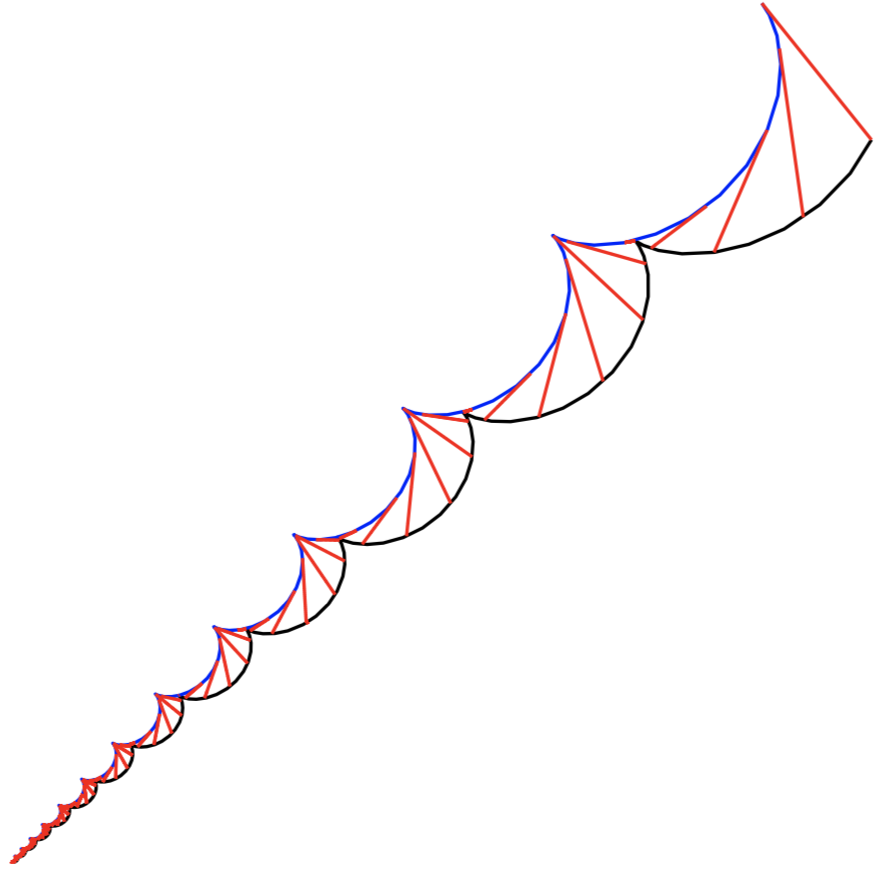}
\hspace{0.5in}
(c) \includegraphics[width=33mm]{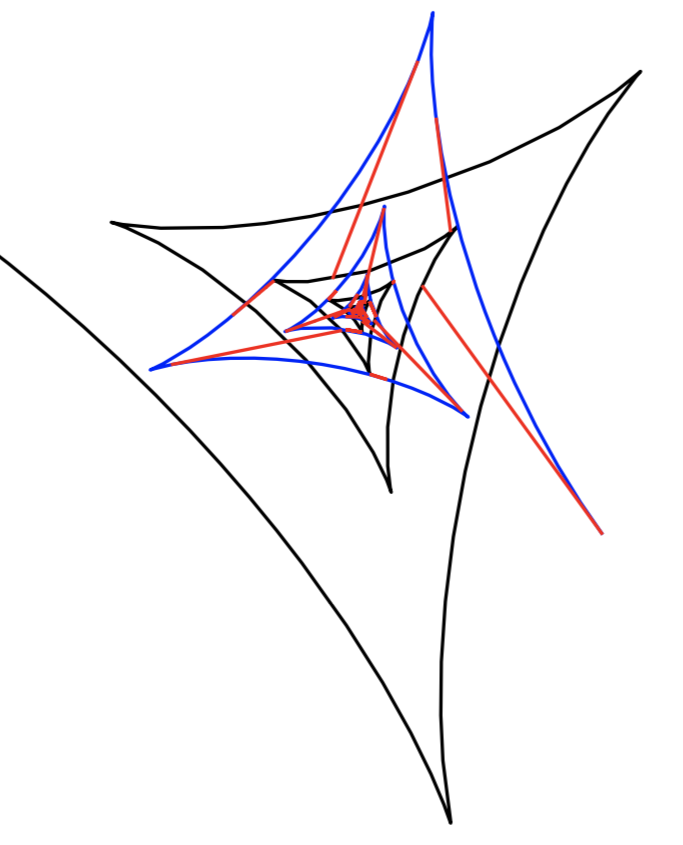}
\par\end{centering}
\vspace{-0.2in}
 \hspace*{-0.1in}\caption{\label{SkewPuis} (a)-(c) Puiseux skew-cycloids and skew-cycloidal spirals, $\phi = \frac{\pi}{6}$.}
\end{figure}
\noindent One has to select the coefficients $A_k$ and $B_k$ carefully to ensure convergence, but finite sums already give an infinite dimensional family of curves. It is impossible to survey them all, but even keeping a single term with complex $\lambda_k$ gives some interesting examples. They are a cross of logarithmic spirals and cycloidal curves. Puiseux, who first found them in the evolute case, described a curve with the inclination equation $R(\theta)=e^{c\theta}\sin(\gamma\theta)$ as follows \cite{Puis}:
{\small \begin{quote} It presents, like the epicycloid, a sequence of cusp points corresponding to the values of $\theta$ which grow in degrees equal to $\pi/\gamma$; but these points, instead of being situated on the same circle, are on a logarithmic spiral, and the tangent at each of them, instead of being confined to the radius vector, makes a constant angle with it. Moreover, the radius vector at each point attains its maximum or minimum, according to whether $\sqrt{c^2+\gamma^2}$ is greater or smaller than $1$.
\end{quote} }
\noindent Puiseux does not say what happens when $\sqrt{c^2+\gamma^2}=1$, see Figure \ref{CuspSpir}(b). This curve resembles a cycloid, with cusps situated on a line, but with increasing arc sizes. As one can see from Figure \ref{CuspSpir}, like logarithmic spirals, Puiseux's cuspidal spirals are self-similar -- the same pattern is reproduced when scaling them. This property holds also for their skew versions, see Figure \ref{SkewPuis}.

\section{Mirrors and caustics by reflection}\label{Mirrors}

After conormal caustics with constant tilting angles the next ones in complexity have tilting angles linear in $\theta$. Caustics by reflection of parallel rays provide one such example. Tschirnhaus and Bernoulli considered caustics of rays coming from a point source of light reflected in a curved mirror. Their shape depends on the relative positions of the source and the mirror, in addition to the shape of the mirror. We will move the source to infinity along a fixed direction to have them depend on the direction and the shape of the mirror only. As a result, the incoming rays will become parallel to this direction, and the reflected rays will all be tangent to the caustic by reflection that we consider \cite[12.3]{Gibs}. For definiteness, we will have the rays incoming along the positive direction of the $x$-axis, and place the mirror ``vertically", see Figure \ref{Mirror}(b). The next definition describes optically feasible mirrors, which do not obstruct their parts from getting hit by incoming rays.
\begin{definition}\label{VertMir} We call a plane curve {\bf vertical} if it is the graph of a function $x=f(y)$. Its {\bf caustic by reflection} is the envelope of the family of rays that are reflections of horizontal rays by it.
\end{definition}
\noindent Thinking of caustics by reflection as conormal caustics, in the notation of Section \ref{Caustic} we have a variable tilting angle, $\phi=\frac{\pi}2-\theta$, see Figure \ref{Mirror}(a). Therefore, $\sin\phi=\cos\theta$, $\cos\phi=\sin\theta$, $\phi'=-1$, and $\phi''=0$, so formula \eqref{R1Gen} simplifies to
\begin{equation}\label{R1Ref}
R_1=\frac14(3\cos\theta\,R(\theta)+\sin\theta\,R'(\theta))\,.
\end{equation}
\begin{figure}[!ht]
\vspace{-0.1in}
\begin{centering}
(a) \includegraphics[width=1.75in]{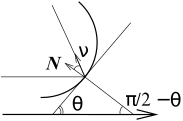}\hspace{0.8in} 
(b) \includegraphics[width=1.5in]{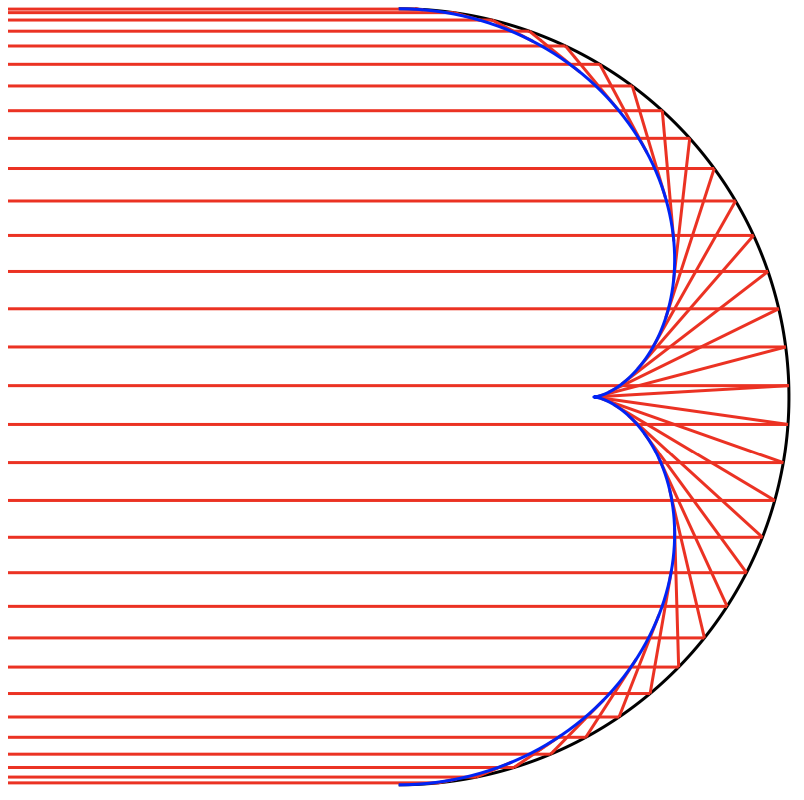}
\par\end{centering}
\vspace{-0.2in}
\hspace*{-0.1in}\caption{\label{Mirror} (a) Conormal vector and the tilting angle of a reflected horizontal ray; (b) reflection diagram and caustic by reflection of a semi-circle (half of nephroid).}
\end{figure}
As a warm-up example, let us apply it to a circular mirror of unit radius, i.e. $R(\theta)=1$. Then $R_1=\frac34\cos\theta$. Of course, only the semicircle $0\leq\theta\leq\pi$ forms a vertical mirror, other parts of the circle would block the light rays coming from the left. In contrast to evolutes and skew-evolutes, the angle of inclination of the caustic is not just shifted by a constant from the angle of inclination of the mirror curve, instead $\theta_1=\theta+\frac{\pi}2-\phi=2\theta$. Therefore, the inclination equation of the caustic is $R_1(\theta_1)=\frac34\cos\frac{\theta_1}2$. This curve is easily recognized as an epicycloid. This particular one is called nephroid (from Greek {\it nephros}, kidney), and it is roughly the lit curve one often sees on the surface of a coffee cup, sometimes called the coffee cup caustic, Figure \ref{Mirror}(b). The reason for ``roughly" is that the light source is usually not the Sun but a light bulb, which is not ``at infinity". We can see that infinite curvature corresponds to a cusp on the caustic at $\theta_1=\pi$.

There is also a simple relation between parametric equations of vertical mirrors and their caustics.
\begin{corollary}\label{MirCaust} 
Let $r(\theta)$ be the radius vector of a vertical mirror and $c(\theta)$ of its caustic. Then
\begin{align}\label{ParCaust}
c(\theta)=r(\theta)+\frac12R(\theta)\sin\theta\,(\cos2\theta,\sin2\theta).
\end{align}
In particular, the length of the ray segment between the mirror and the caustic is $\frac12|R(\theta)\sin\theta|$.
\end{corollary}
\begin{proof} Since $\phi=\frac{\pi}2-\theta$ we have by trigonometry
\begin{align*}
\nu = \sin\phi\,T+\cos\phi\,N=\sin\phi\,(\cos\theta,\sin\theta)+\cos\phi\,(-\sin\theta,\cos\theta)\\
=(\sin(\phi-\theta),\cos(\phi-\theta))=(\cos2\theta,\sin2\theta).
\end{align*}
The desired formula follows directly from formulas  $c=r+\frac{\cos\phi}{\chi}\,\nu$ and $\chi=\k - \dot{\phi}=\frac{1-\phi'}R$ from the proof of Theorem \ref{CaustInc}. The claim about the length follows from the distance formula.
\end{proof} 
It will be convenient for us in many cases to extend mirror curves beyond the angles of inclination $0\leq\theta\leq\pi$, to all $\theta\geq0$. The next corollary gives us simple facts about the shape of vertical mirrors and their caustics.
\begin{corollary}\label{CuspLoc} Vertical mirrors without inflection points consist of concave or convex arcs separated by horizontal cusps. If their caustics by reflection are also vertical they have horizontal cusps at the same points as the cusps of the mirror, and at the points corresponding to the points of mirror's vertical tangency. The latter are on the same horizontal segment, removed from them by half the radius of mirror's curvature there. 
\end{corollary}
\begin{proof} In the absence of inflection points $R$ can only change sign at cusps, and non-horizontal cusps are inconsistent with verticality. Hence $R(\theta)$ keeps the same sign between multiples of $\pi$ (where the cusps are), and the mirror is either convex or concave between them. Conversely, if $R$ switches signs at multiples of $\pi$ then $R\sin\theta$ is positive or negative for all $\theta$, and $y(\theta)$ is a monotone function of $\theta$ by \eqref{Inc-Par}. Hence $x(y)$ is a function graph.

Since the rays reflected at horizontal or vertical points on the mirror remain horizontal, and the caustic is tangent to them, it must have horizontal tangents at those points. And since the nearby rays are reflected to the opposite sides of this tangent the unit tangent to the caustic reverses direction at them, i.e. the caustic has cusps there. When $\theta=\pi n$ we have $\sin\theta=0$ and \eqref{ParCaust} shows that $c(\pi n)=r(\pi n)$. When $\theta=\pi\big(n+\frac12\big)$ the $y$ coordinates of $c$ and $r$ are the same because of the $\sin2\theta$ factor, and the horizontal distance is $\frac12|R|$ because $|\sin\pi\big(n+\frac12\big)|=1$.
\end{proof} 
\noindent These facts are illustrated by the example of circle and nephroid in Figure \ref{Mirror}(b). In particular, if the mirror has locally maximal or minimal curvature at the vertical points (as in the case of the circle) then the corresponding cusps of the caustic bisect the horizontal segments connecting them to the mirror's centers of curvature at those points. Note, however, that vertical mirrors are not obliged to have vertical caustics, the latter can even have self-intersection points \cite{BGG}. 

From now on we will assume that our mirrors are concave (between cusps) in the direction of incoming rays, so that their caustics are visible rather than virtual. Let us now turn to designing mirrors whose caustics reproduce their shape. For caustics by reflection the similarity equation \eqref{SimEq} simplifies to
\begin{equation}\label{RefSimi}
\sin\theta\,R'(\theta)+3\cos\theta\,R(\theta)=4aR(\pm(2\theta-\beta)).
\end{equation}
In contrast to evolutes and skew-evolutes, even point by point case $\beta=0$ does not lead to a first order ODE here. The only exception is when $a=0$ and the caustic is a single point. Integrating it we find $R(\theta)=\frac{A}{\sin^3\theta}$. The corresponding parametric equations (without additive constants) are, according to \eqref{Inc-Par}:
$$
x(\theta)=\int R(\theta)\cos\theta\,d\theta=-\frac{A}{2\sin^2\theta};\ \ \ \ \ \ \
y(\theta)=\int R(\theta)\sin\theta\,d\theta=-A\cot\theta\,.
$$
From trigonometric identities the equation in $x$-$y$ coordinates is $y^2+2Ax=-A^2$, which is a family of horizontal parabolas with {\sl latus rectum} $2A$. Note that we do not need to extend $\theta$ beyond $(0,\pi)$ for the parabolic mirror, this interval already gives the entire parabola. The property of parabolas to focus a beam of parallel light rays at a single point is well-known from classical geometry, see Figure \ref{CataCaust}(a), not the least because of the semi-legendary story of Archimedes burning the Roman fleet with a parabolic mirror. But it is satisfying to find the shape by honest toil, rather than by simply recalling a known property of parabolas. 

From now on we will restrict to the point by point similarity. This means that we want to solve the equation 
\begin{equation}\label{RefPant}
\sin\theta\,R'(\theta)=4aR(2\theta)-3\cos\theta\,R(\theta),
\end{equation}
which is neither ordinary nor even delay differential equation that we encountered earlier. An equation of this sort, $y'(t)=y(\lambda t)$, first appeared in Mahler's work on  partitions in 1940, but was not much studied until 1970-s, when Ockendon and Tayler derived a more general form with constant coefficients, $y'(t)=ay(\lambda t)+by(t)$. It described a model of a current collection system, the so-called {\it pantograph}, that one sees on top of electric trains connecting them to the wires above \cite{Fox}. Since 1990-s functional differential equations with arguments of the form $\lambda t$, even with multiple $\lambda$-s and variable coefficients, came to be called pantograph equations, see \cite{DGT,M-PN} for recent reviews. But most of examples considered in the literature have constant coefficients, and even \cite{M-PN}, that does allow variable coefficients, assumes that they are non-singular at $\theta=0$, i.e. the coefficient in front of the derivative does not vanish there. This is the worst case scenario analytically, and little is known about solutions to such equations.
\begin{definition}\label{MirPant} We call  \eqref{RefPant} the {\bf mirror pantograph equation} and curves corresponding to its solutions  {\bf pantograph mirrors}.
\end{definition}
\begin{figure}[!ht]
\centering
(a) \includegraphics[width =45mm,angle=90,origin=c]{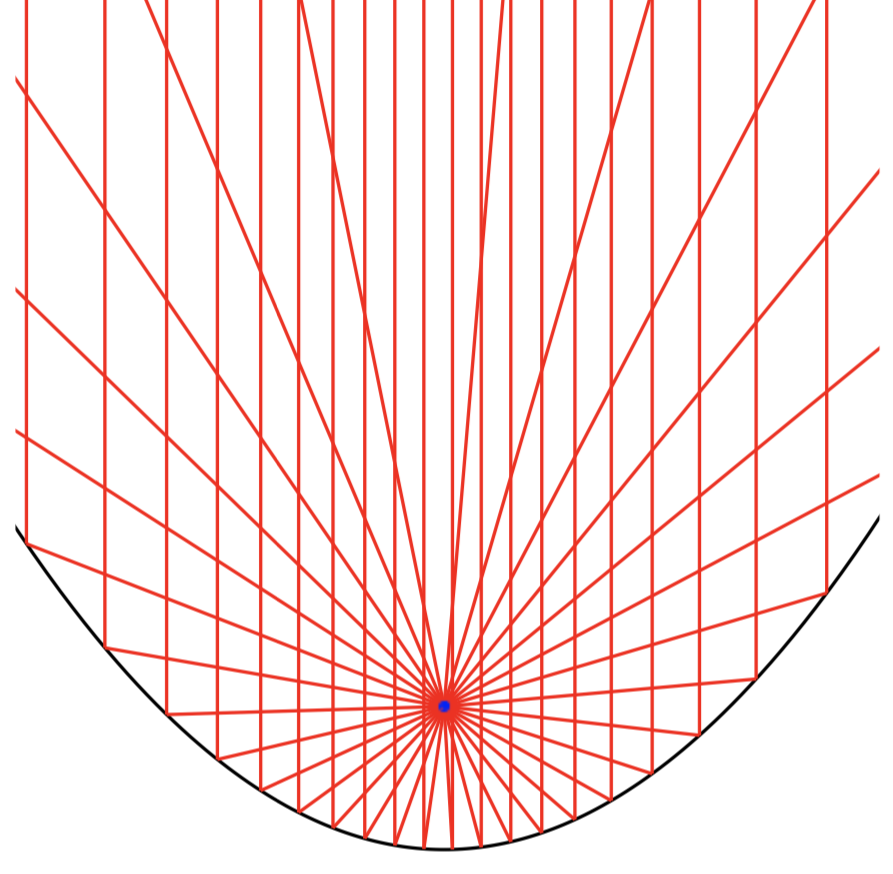}
\hspace{0.7in}
(b) \includegraphics[width =45mm,origin=c]{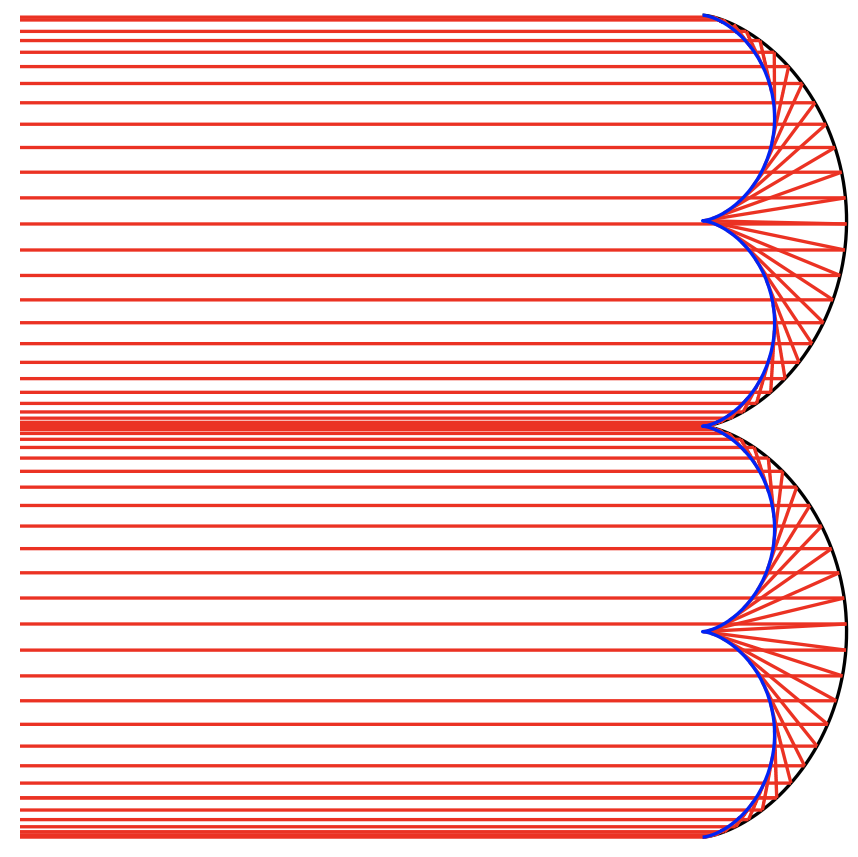}
\caption{\label{CataCaust} Reflection diagrams and caustics by reflection of a (a) parabolic mirror; (b) cycloidal mirror.}
\end{figure}

There is more bad news. Assuming a smooth concave mirror we expect it to have a vertical tangent at some point, which means the caustic must have a cusp at the corresponding point (excluding parabolic mirrors when the entire caustic collapses into a single point). But how can a mirror with no cusps have a congruent or similar caustic that has them? It seems that \eqref{RefPant} has no geometrically meaningful solutions at all! But this is too hasty. The mirror can have cusps at the edges, i.e. at $\theta=0,\pi$, and an arc of the caustic between its cusps could reproduce the mirror's shape. If we extend the mirror indefinitely up, as we did, perhaps it can be similar to its caustic even strictly speaking. We will show that this is indeed the case.

Consider an auxiliary equation obtained from  \eqref{RefPant} by substituting $R=Q\sin\theta$:
\begin{equation}\label{RefPantQ}
\tan\theta\,Q'(\theta)=8aQ(2\theta)-4Q(\theta).
\end{equation}
It is easy to see that $Q=\const$ is a solution for $a=\frac12$, and up to scale this gives $R=\sin\theta$. This is the classical cycloidal mirror \cite{McL}, see Figure \ref{CataCaust}(b). It gives, in a sense, the ideal configuration for a mirror similar to its caustic. Within each arc of the cycloid exactly two arcs of the caustic's half-sized cycloid fit. Extending the mirror and the caustic up indefinitely we get two globally similar curves.
\begin{figure}[!ht]
    \centering
\includegraphics[width=150mm]{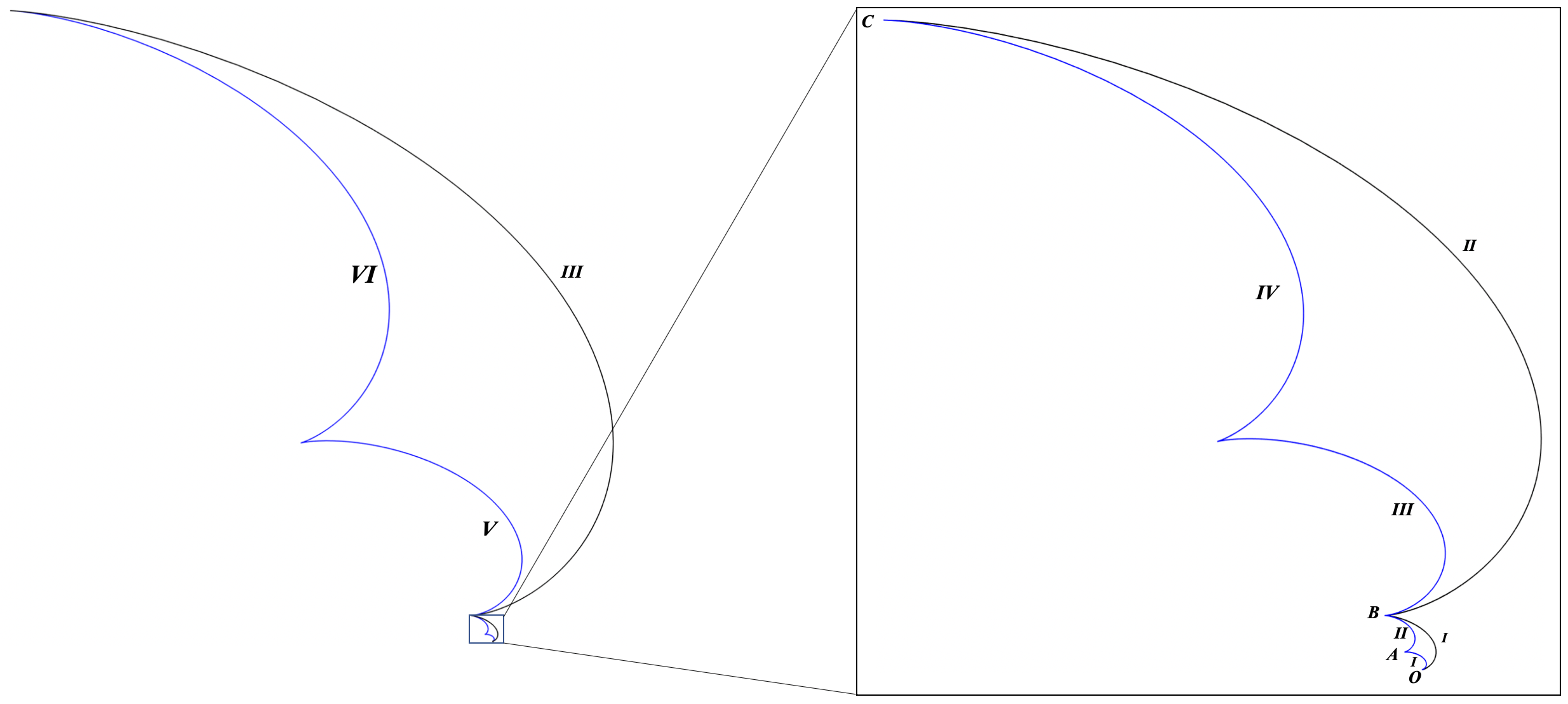}\hspace{0.7in}
\caption{Similarity structure of a pantograph mirror and its caustic: identically labeled arcs are similar to each other.}
    \label{ScaledArc}
\end{figure}

Are there mirrors other than the cycloidal mirror that are similar to their caustics like this? At first glance, it seems that the answer must be no. We certainly can not fit three or more caustic arcs within a single arc of a concave mirror. Each cusp would have to correspond to a vertical tangent to the mirror, and there is only one of those on it. Thus, we expect a chain of concave arcs with cusps at the ends, each of which generates a caustic with two arcs separated by a cusp, whose other cusps coincide with the cusps of the mirror by Corollary \ref{CuspLoc}, see Figure \ref{ScaledArc}. Here is a simple observation about the shape of such mirrors.
\begin{corollary}\label{CuspLine} 
Given a vertical pantograph mirror with horizontal tangent at $\theta=0$ the cusps of its caustic corresponding to the mirror's angles of inclination $\theta=0,\frac\pi2,\pi,2\pi,4\pi,\dots$ lie on the same line.
\end{corollary}
\begin{proof} Place the origin $O$ at $r(0)$ and label the listed cusps $A$, $B$, $C$, etc, see Figure \ref{ScaledArc}. By assumption, arcs I of the mirror and the caustic start at the origin and are similar without rotation or reflection. Hence the segment $OB$ is obtained by central scaling of $OA$ with the scaling factor $a^{-1}$, and points $O$, $A$, $B$, corresponding to $\theta=0,\frac\pi2,\pi$, are on the same line. By the same argument applied to arcs II, $A$, $B$, $C$, corresponding to $\theta=\frac\pi2,\pi,2\pi$, are also on the same line. Hence all four points are on the same line, and the general conclusion follows by induction.
\end{proof} 
Equation \eqref{RefPant} means that the first arc of the caustic is similar to the first arc of the mirror, second to the second, and so on. In the case of the cycloid {\it both} initial arcs of the caustic were similar to the first arc of the mirror that generates them (and hence to each other). This then implied that {\it all} arcs of the mirror and of the caustic are similar, and all cusps are on the same line, the $y$-axis. This might have been a lucky accident, and too much to ask for in general. Thus, in the next section we will study  equation \eqref{RefPantQ} and non-cycloidal mirrors with a weaker similarity structure sketched on Figure \ref{ScaledArc}. 

\section{Solving the mirror pantograph}\label{Panto}

We already converted the mirror pantograph equation \eqref{RefPant} into a simpler form \eqref{RefPantQ}. To get some intuition about the solutions to the latter let us consider a simplified equation that approximates it near $0$:
\begin{equation}\label{LinPantQ}
\theta Q'(\theta)=8aQ(2\theta)-4Q(\theta).
\end{equation}
Conjecturing a solution of the form $Q(\theta)=\theta^k$, we get the characteristic equation $k=2^{k+3}a-4$. Hence such solutions exist only for special similarity factors $a=\frac{k+4}{2^{k+3}}$. Indeed, one can show more. Assuming that $Q(\theta)$ has a Laurent expansion at $0$, the equation implies that it must be a single power $\theta^k$, with $k$ an integer, or a linear combination of powers with the same $a$ (this only happens for $k=-3,-2$ with $a=1$). Thus, for $a=\frac{k+4}{2^{k+3}}\neq1$ with an integer $k$ we should expect a single solution to \eqref{RefPantQ} (up to scale) that behaves like $\theta^k$ at $0$.
To prove it, we recall the Taylor expansion of $\tan\theta$ at $0$:
\begin{equation}
    \tan\theta=\tau_0\theta+\tau_2\,\theta^3+\tau_4\,\theta^5+\dots=\sum_{n=0}^{\infty}\tau_{2n}\,\theta^{2n+1},
\end{equation}
where $\tau_{2n}=2\frac{2^{2n}-1}{\pi^{2n}}\,\zeta(2n)$, and $\zeta$ is the Riemann zeta function $\ds{\zeta(s):=\sum_{n=1}^\infty \frac{1}{n^s}}$. 
Obviously, $\zeta(s)$ is monotone decreasing for real $s>1$, and since $\zeta(2)=\frac{\pi^2}{6}$ we have for $k\geq1$:
\begin{equation}\label{TanEst}
0<\tau_{2n}\leq\frac{\pi^2}{3}\left(\frac{2}{\pi}\right)^{2n}.    
\end{equation}
Now we are ready to construct solutions to \eqref{RefPantQ} analogous to $\theta^k$ for \eqref{LinPantQ}.
\begin{lemma}\label{SolQ} Given an integer $k$, there exists an analytic (in a punctured neighborhood of $0$) solution $Q(\theta)=\sum_{n=k}^{\infty}a_n\,\theta^n$ to \eqref{RefPantQ} with $a_k\neq0$ if and only if $a = \frac{k+4}{2^{k+3}}$. Among analytic solutions, this solution is unique up to scale unless $k=-3$, in which case there is a two-parameter family of solutions. Moreover, all $a_n$ with $n$ of the same parity as $k$ have the same sign as $a_k$, the rest are $0$, and $|a_n| \leq M\left(\frac{2}{\pi}\right)^{n}$ for some $M>0$. In particular, the Laurent series for $Q(\theta)$ converges for $|\theta|<\frac\pi2$.
\end{lemma}
\begin{proof}
We will look for a solution in the form  $Q(\theta)=\sum_{n=k}^{\infty}a_n\,\theta^n$ with $a_{k}\neq0$. Substituting this series and the Taylor series for $\tan\theta$ at $0$ into \eqref{RefPantQ}, and equating the coefficients of $\theta^n$ we obtain the equation:
\begin{equation}\label{anrecurs}
    (2^{n+3}a-4)\,a_n = 
    na_n + \tau_2(n-2)\,a_{n-2} + 
    \tau_4(n-4)\,a_{n-4} + \dots
\end{equation}
The sum is finite because $a_m=0$ for $m<k$.
Note that $(2^{k+3}a-4-k)a_k=0$, so $\ds{a = \frac{k+4}{2^{k+3}}}$. There are no further restrictions on $a_k$, so it can be selected arbitrarily. The same equation holds with $k$ replaced by $k+1$, and it implies that $a_{k+1}=0$, unless $k=-3$ and $a=1$. In that case $a_{-2}$ is also arbitrary. Since $a_{k+1}=0$ all  coefficients with indices of the opposite parity are $0$ by \eqref{anrecurs}. Solving \eqref{anrecurs} formally for $a_n$ we get the recursion
\begin{equation}\label{anRec}
a_n=\frac{\tau_2(n-2)a_{n-2}+\tau_4(n-4)a_{n-4}+\dots}{2^{n+3}\,a-n-4} = \frac{1}{2^{n+3}a-n-4}\sum_{i=1}^{\lfloor\frac{n-1}{2}\rfloor} \tau_{2i}(n-2i)\,a_{n-2i}.
\end{equation}
The denominator can be rewritten as $\ds{2^{n+4}\left(\frac{a}2-\frac{n+4}{2^{n+4}}\right)}$. By elementary calculus, $f(t):=\frac{t+4}{2^{t+4}}$ is strictly monotone decreasing for $t+4>\frac1{\ln2}$, and $f(k)=\frac{a}2$ by assumption on $a$. Hence $\frac{a}2-f(n)>0$, and the denominator is strictly positive for $n>k$. 

Thus, the recursion \eqref{anRec} is well defined, and implies that all $a_n$ of the same parity as $k$ are non-zero and have the same sign. It remains to show that the series with the coefficients the recursion determines converges. Pick $N$ large enough to make   
$$
\ds{\dfrac{(n-\lfloor\frac{n}{2}\rfloor)\lfloor\frac{n}{2}\rfloor}{2^{n+4}\,a-n-5}\frac{\pi^2}{3}\leq1}
$$ 
for $n\geq N$. Denote $R = \frac{\pi}{2}$ and $A=\frac{\pi^2}{3}$, and set $\ds{M:=\max_{n \leq N} \left(|a_n|R^{n}\right)}$. By definition of $M$, we have $|a_n| \leq \dfrac{M}{R^{n}}$ for $n \leq N$, and by  \eqref{TanEst}, $|\tau_{2n}|\leq\dfrac{A}{R^{2n}}$. We proceed by induction assuming that the inequality holds up to $n=N$. By the recursion:
\begin{align*}
    |a_{n+1}| &\leq \frac{1}{2^{n+4}a-n-5}\sum_{i=1}^{\lfloor\frac{n}{2}\rfloor} |\tau_{2i}||n+1-2i||a_{n+1-2i}| \\
    &\leq \dfrac{1}{2^{n+4}a-n-5} \sum_{i=1}^{{\lfloor\frac{n}{2}\rfloor}}(n+1-2i)\dfrac{A}{R^{2i}}\cdot\dfrac{M}{R^{n+1-2i}}\\
    &\leq\dfrac{1}{2^{n+4}a-n-5}\dfrac{AM}{R^{n+1}}\sum_{i=1}^{{\lfloor\frac{n}{2}\rfloor}}(n+1-2i)
=\dfrac{(n-\lfloor\frac{n}{2}\rfloor)\lfloor\frac{n}{2}\rfloor A}{2^{n+4}a-n-5}\cdot\dfrac{M}{R^{n+1}}.
\end{align*}
By definitions of $N$ and $A$, the coefficient in front of $\ds{\frac{M}{R^{n+1}}}$ is $\leq1$, and the induction step is concluded.
\end{proof}
\noindent Note that if $Q\sim\theta^k$ with $k<-1$ then $R=Q\sin\theta$ will have a pole at $0$, indeed, the parabolic mirror corresponds to $k=-4$. We will now restrict our attention to $k\geq-1$ and convert solutions to the auxiliary equation \eqref{RefPantQ} into those for the original one \eqref{RefPant}. The idea is to use the doubling of the argument in the pantograph equation to extend the solution from a neighborhood of zero to the entire half-axis $\theta\geq0$.
\begin{theorem} For any integer $m\geq0$, there exists a global analytic solution $R(\theta)$ to \eqref{RefPantQ} on $[0,\infty)$ with $R(\theta)\sim\theta^m$ near $0$ if and only if $\ds{a = \frac{m+3}{2^{m+2}}}$. This solution is unique up to scale.
\end{theorem}
\begin{figure}[!ht] 
\centering
(a) \includegraphics[width =40mm,origin=c]{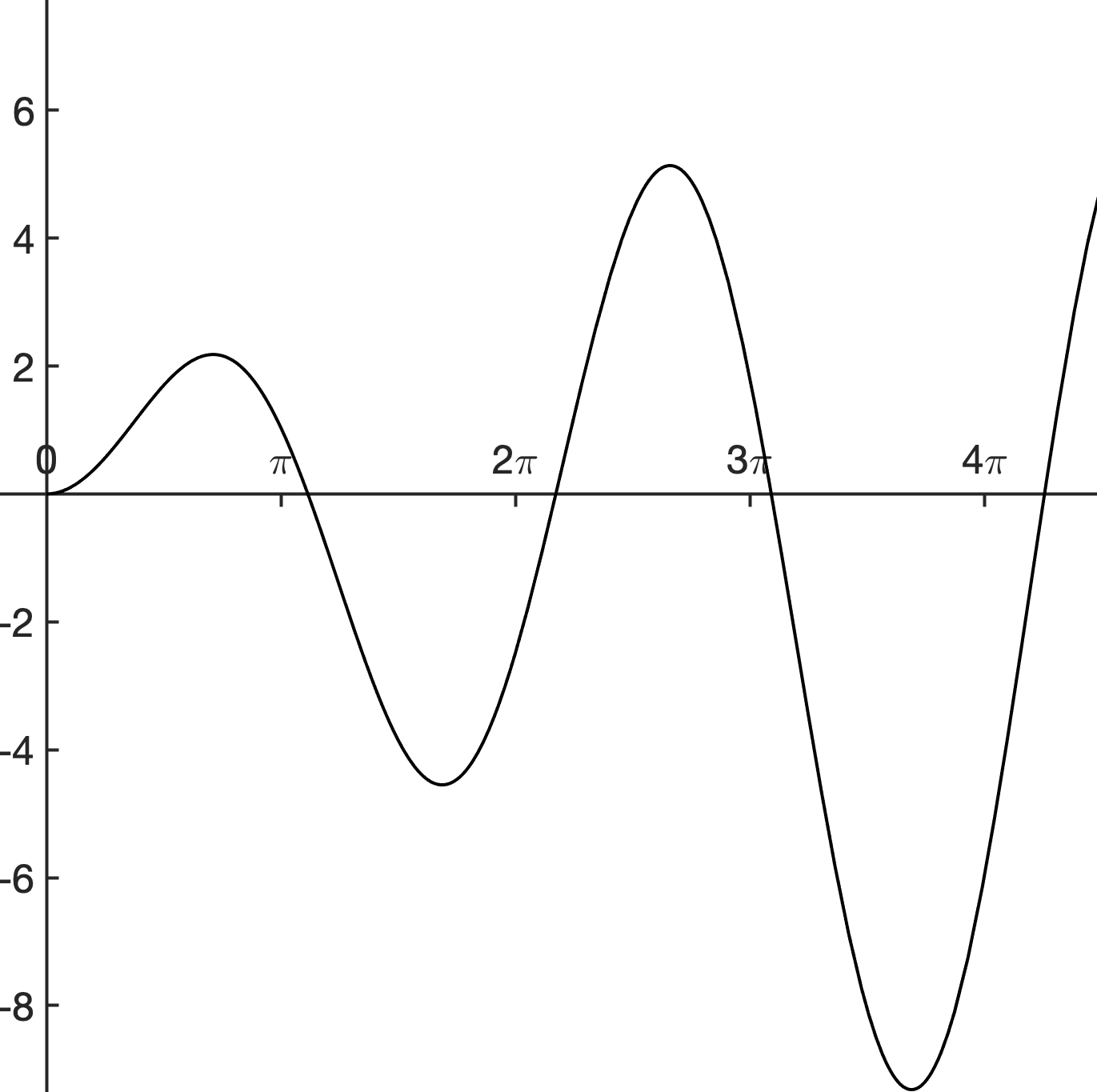}
\hspace{0.2in}
(b) \includegraphics[width =40mm,origin=c]{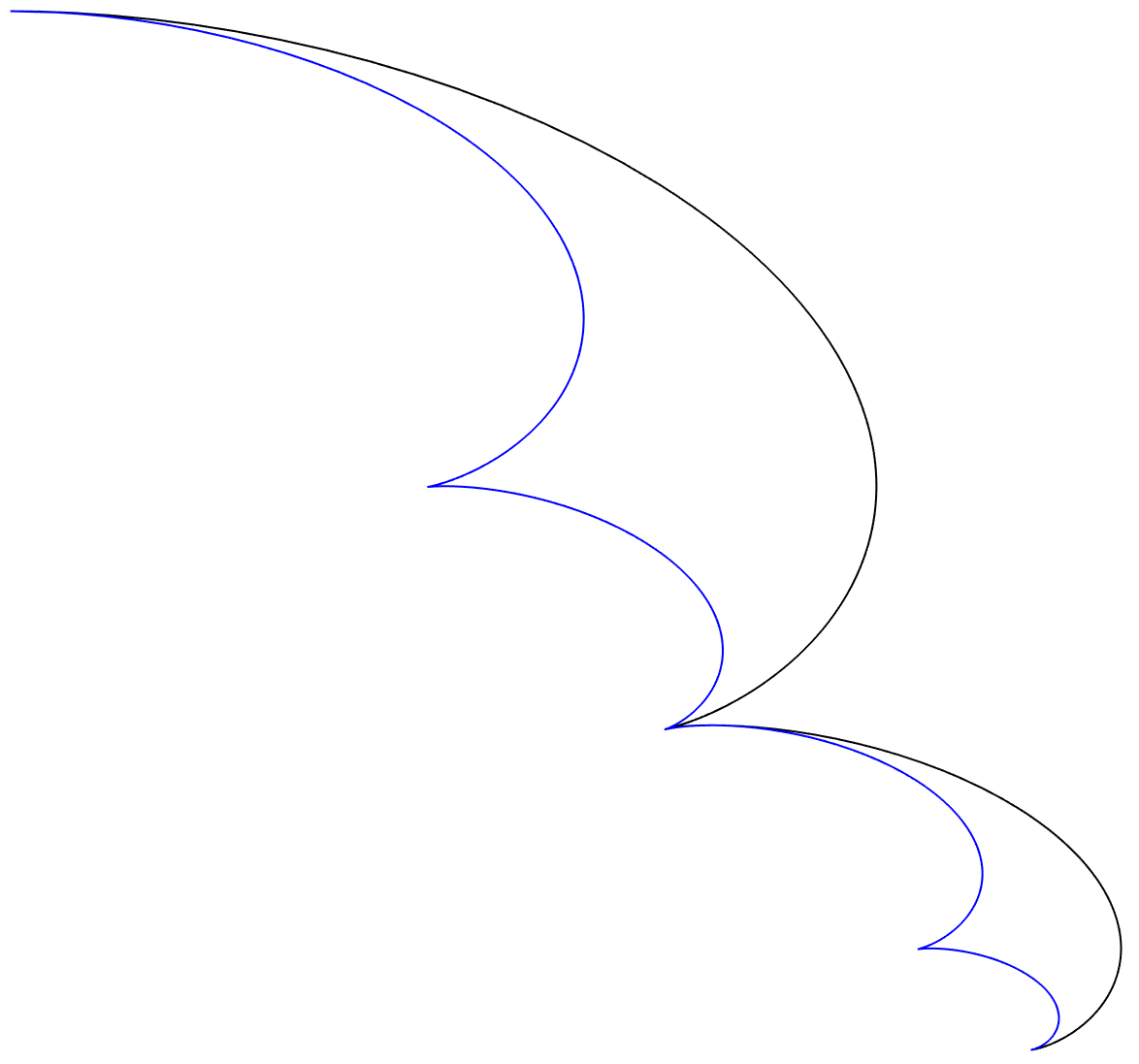}
\hspace{0.2in}
(c) \includegraphics[width =40mm,origin=c]{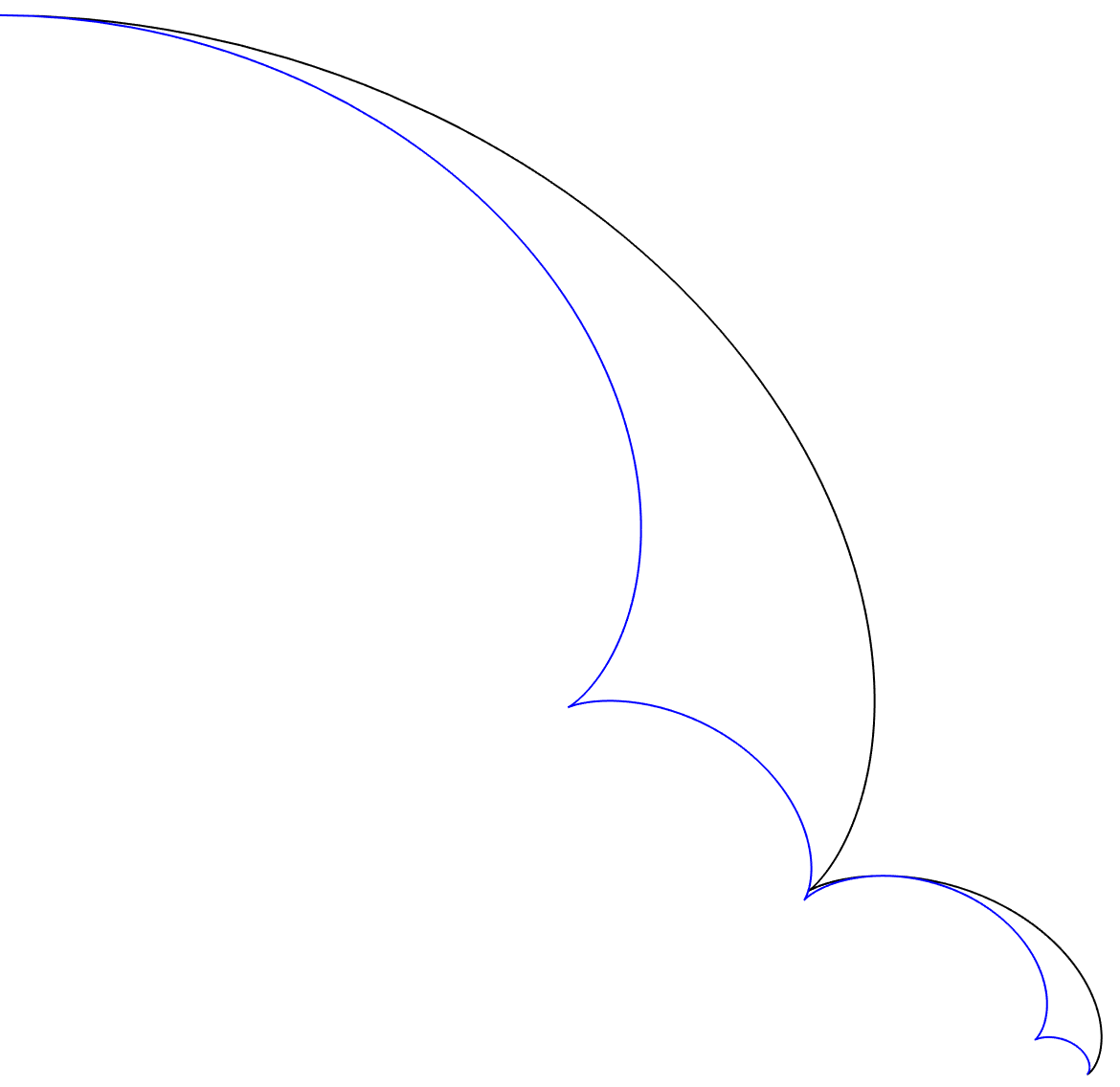}
\caption{\label{RDoubling} Analytic pantograph mirrors similar to their caustics: (a) graph of $R(\theta)$ for $m=2$, $a=\frac5{16}$; (b) mirror and caustic for $m=2$, $a=\frac5{16}$; (c) for $m=3$, $a=\frac3{16}$.}
\end{figure}
\begin{proof}
By Lemma \ref{SolQ} we have such a solution $R(\theta)=Q(\theta)\sin\theta$ on $[0,\frac\pi2)$. Rearranging the terms, we can rewrite \eqref{RefPant} as
\begin{equation}\label{RefDoub}
R(2\theta)=\frac1{4a}\Big(3\cos\theta\,R(\theta)+\sin\theta\,R'(\theta)\Big)
\end{equation}
As $\theta$ varies over $\big[0,\frac\pi2\big)$ on the right $2\theta$ varies over $[0,\pi)$. The values of $R(\theta)$ on $\big[0,\frac\pi2\big)$ agree with the original ones because it satisfies \eqref{RefPant} there by construction. Hence the doubling formula \eqref{RefDoub} provides an analytic continuation of $R(\theta)$ to $[0,\pi)$, and the continued function satisfies the mirror pantograph equation there. Repeating the process we can continue it to $[0,2\pi)$, $[0,4\pi)$, and so on. By induction, we obtain an analytic solution on $[0,\infty)$.
\end{proof}

Figure \ref{RDoubling} shows the analytic solution generated numerically for $m=2$, $a=\frac5{16}$. For $\theta<\frac\pi2$ we used the power series for $Q(\theta)$ with the coefficients determined recursively from \eqref{anRec} (up to $n=30$), and $R(\theta)=Q(\theta)\sin\theta$. For $\theta\geq\frac\pi2$ we found $R(\theta)$ by iterating the doubling formula \eqref{RefDoub}. At first glance, everything looks in order. But, as one can clearly see on the graph, the computed zeros of $R(\theta)$ deviate from the multiples of $\pi$. As a result, the cusps of the mirror and its caustic are not exactly horizontal, and the effect becomes more pronounced as $m$ increases and $a$ decreases, see Figure \ref{Cata}.
\begin{figure}[!ht] 
\centering
(a) \includegraphics[width =35mm,height=40mm,origin=c]{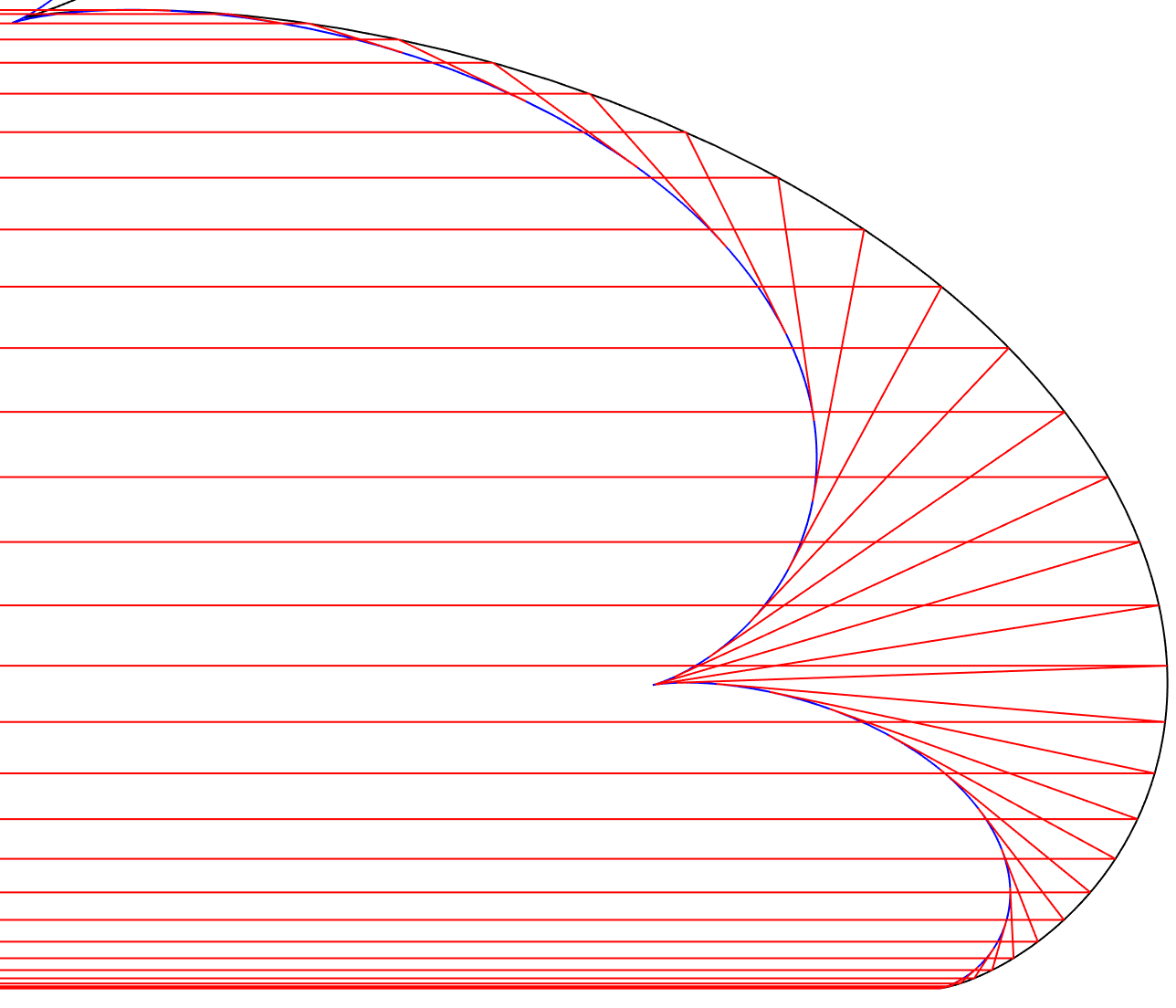}
\hspace{0.7in}
(b) \includegraphics[width =35mm,height=40mm,origin=c]{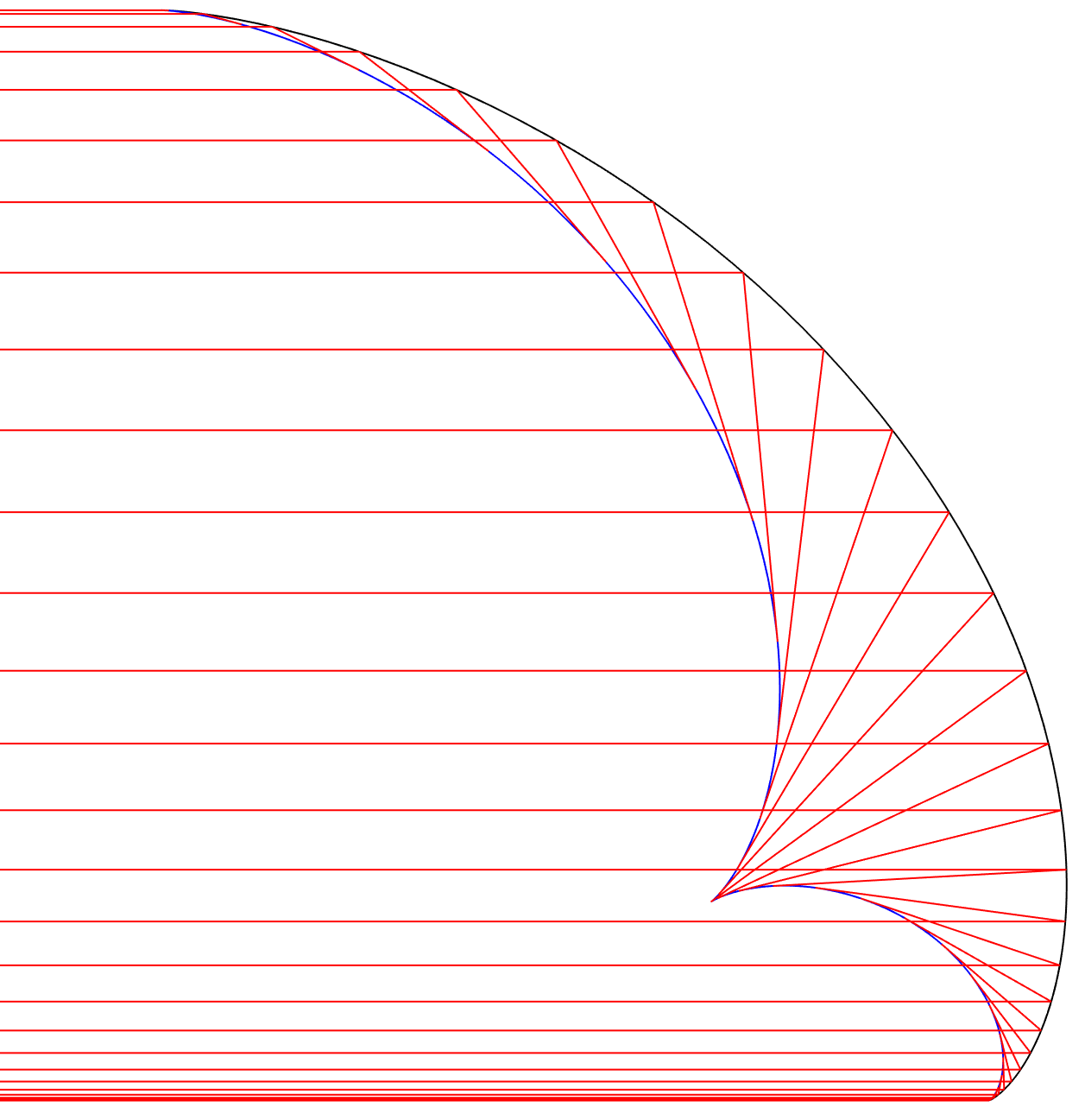}
\caption{\label{Cata} Reflection diagrams of analytic pantograph mirrors and their caustics, middle cusps tilt down from the horizontal: (a) $m=2$, $a=\frac{5}{16}$; (b) $m=3$, $a=\frac{3}{16}$. }
\end{figure}
The cusps are not quite on the same line either, although that is harder to detect visually, it is more visible for $m=3$, Figure \ref{RDoubling}(c). 

The reason for this disappointing result is that equation \eqref{RefPant} does not by itself guarantee verticality of the mirror, and hence its physical feasibility (recall that even the caustic of a vertical mirror may not be vertical). We also need $R(\theta)$ to switch signs at multiples of $\pi$. Since $R=Q\sin\theta$ it would suffice that $Q$ is finite at multiples of $\pi$. However, Figure \ref{QRat}(a) shows the graph of $Q=Q(\theta)$ generated using the power series, which suggests that it has a pole at $\pi$. 
\begin{figure}[!ht] 
\centering
(a) \includegraphics[width =40mm,origin=c]{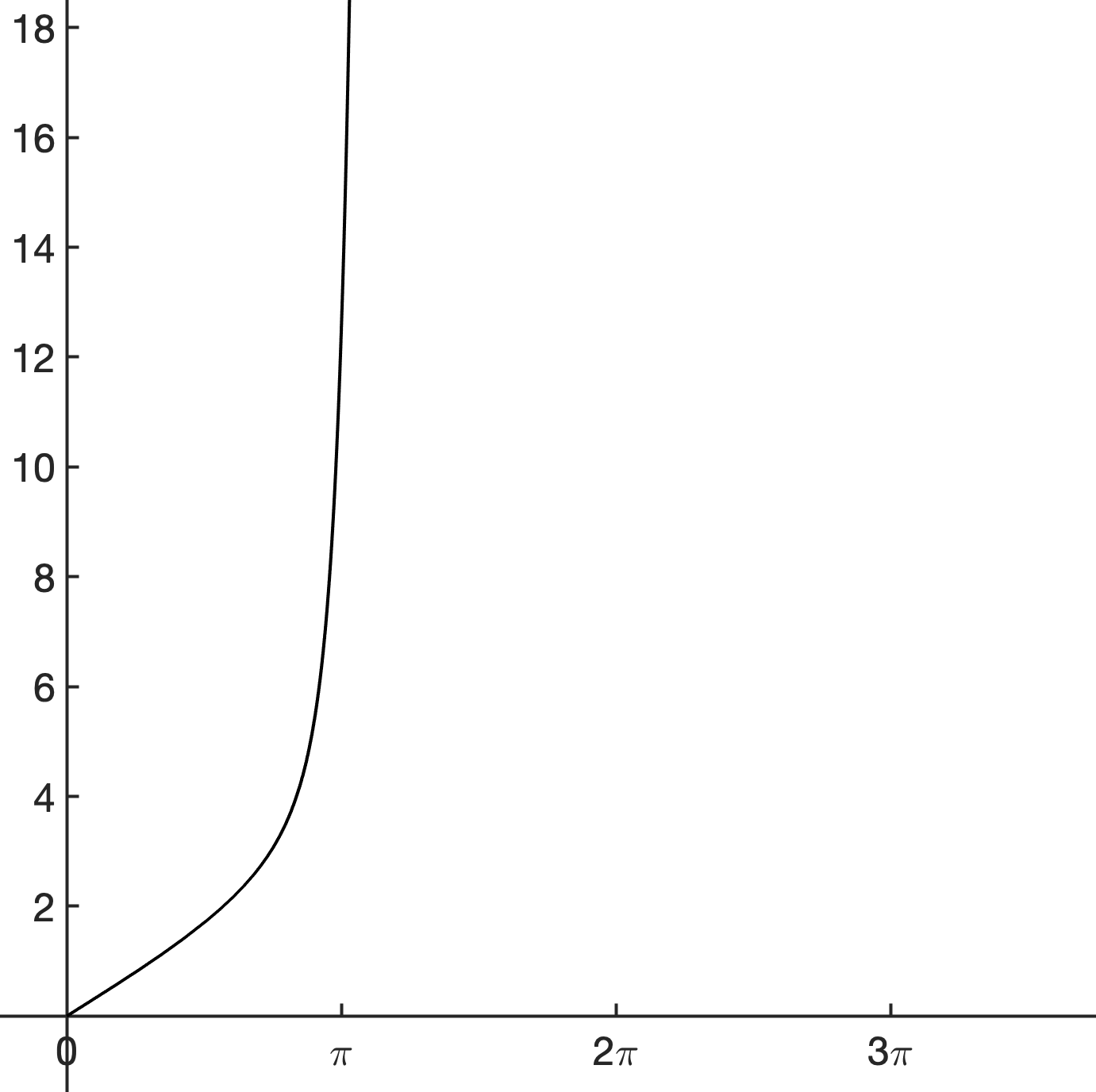}
\hspace{0.9in}
(b) \includegraphics[width =40mm,origin=c]{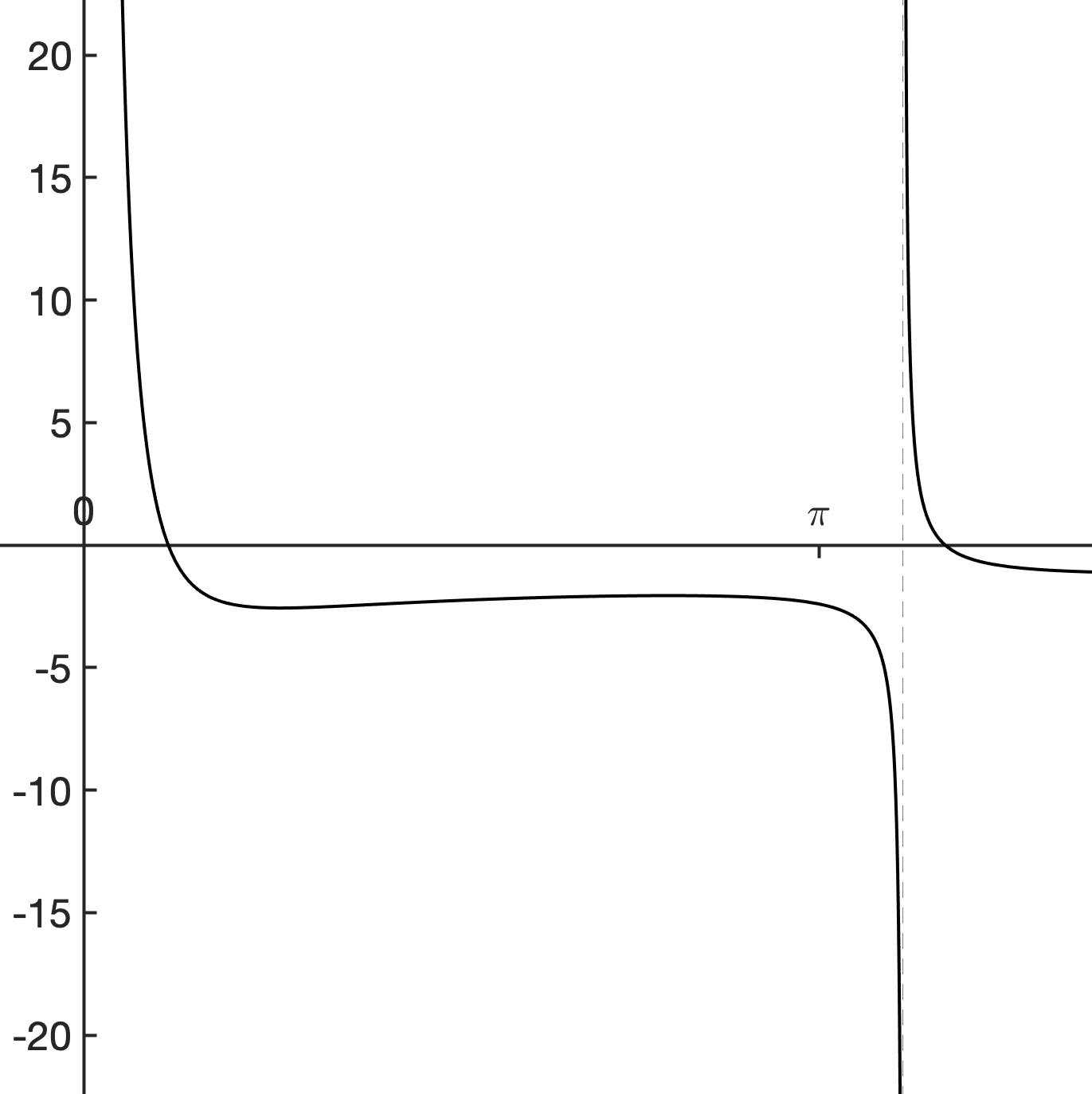}
\caption{\label{QRat} Analytic pantograph solutions for $m=2$, $a=\frac5{16}$: (a) graph of $Q(\theta)$; (b) graph of $\frac{R(\theta+\pi)}{R(\theta)}$.}
\end{figure}

In other words, while the mirror pantograph has analytic solutions other than the cycloid, they are not feasible mirrors. Still, the solution for $a=\frac{5}{16}$ is close to being vertical, and one can ask if there is, perhaps, a {\it non-analytic} vertical solution in its vicinity. On Figure \ref{RDoubling}(b) this pantograph mirror looks similar to its caustic even in a stronger sense we saw with the cycloidal mirror, where all arcs between cusps are similar to each other. For a solution to \eqref{RefPant} to have this property it suffices that the first two arcs of the mirror are similar to each other, i.e. that the radii of curvature at $\theta$ and $\pi+\theta$ are in the same ratio for every $\theta$. On Figure \ref{QRat}(b) we graphed this ratio for the $a=\frac5{16}$ mirror to show that it is not, in fact, constant. However, it does not vary much away from the cusps creating the visual impression of approximate similarity, and so one again wonders if there are non-analytic vertical mirrors that are similar to their caustics in this stronger sense.

\end{document}